\documentclass[
	parskip=half
	, fontsize=11pt,
	, listof=entryprefix]{scrartcl}
\usepackage[
  a4paper,
  left=2.5cm,
  right=2.5cm,
  top=3cm,
  bottom=4cm
]{geometry}

\author{	Christian Bender and Benedikt Simon Flierl\\[8pt]
	\textit{Department of Mathematics, Saarland University}
}
\title{A Deep BSDE Method for a Class of Strongly Coupled FBSDEs}

\usepackage{amsmath, amsthm, amssymb, mathtools}

\let\originalleft\left
\let\originalright\right
\renewcommand{\left}{\mathopen{}\mathclose\bgroup\originalleft}
\renewcommand{\right}{\aftergroup\egroup\originalright}

\newcounter{main}[section]

\theoremstyle{plain}
\newtheorem{theorem}{Theorem}[section]
\newtheorem{lem}[theorem]{Lemma}

\newtheorem{assumption}[theorem]{Assumption}

\theoremstyle{definition}

\newtheorem{remark}[theorem]{Remark}

\newtheorem{example}[theorem]{Example}

\makeatletter
\let\c@figure\c@main

\makeatother

\usepackage{xcolor}
\usepackage{graphicx}

\usepackage{array}
\usepackage{longtable}
\usepackage{arydshln}

\newcommand{\crossone}{\times_{p=1}}
\newcommand{\crosstwo}{\times_{p=2}}
\newcommand{\crossfour}{\times_{p=4}}

\counterwithin*{figure}{section}

\allowdisplaybreaks

\vfuzz \hfuzz
\usepackage[backend=biber, style=numeric, giveninits=true, maxbibnames=999, maxcitenames=2]{biblatex}
\usepackage{hyperref}

\bibliography{references}

\begin{document}
	\maketitle
	\pagestyle{headings}
	\begin{abstract}
		We investigate a variant of the deep BSDE method introduced by \citeauthor{E_Han_Jentzen_2017}~\cite{E_Han_Jentzen_2017,Han_Jentzen_E_2018}. The key novelty is that we establish an a-posteriori convergence result for the approximation of strongly coupled forward-backward stochastic differential equations (FBSDEs), i.\,e., our result holds without any assumptions on small time horizons, monotonicity or weak coupling that are typically imposed in the literature on the deep BSDE method. Instead, we rely on smoothness assumptions on the coefficients and cover FBSDEs in which the coupling of the BSDE into the SDE depends on both the backward component~$Y$ and the control component~$Z$. Numerical experiments illustrate the theoretical results and demonstrate the applicability of the proposed approach.
		
		\medskip
		\noindent\emph{Keywords}: deep BSDE, strongly coupled FBSDE, a-posteriori estimate, convergence
		\medskip
        \noindent\emph{MSC 2020}: 60H10, 60H30, 65C05, 65C30, 68T07 
	\end{abstract}
	
	\section{Introduction}
	In this paper, we study the numerical approximation of coupled forward-backward stochastic differential equations (FBSDEs) of the form
    \begin{align}\label{eq:FBSDE}
        \left\{ \begin{aligned}
            X_t &= x_0+\int_0^tb(s,X_s,Y_s,Z_s)\,\mathrm{d}s+\int_0^t\sigma(s,X_s,Y_s)\,\mathrm{d}W_s,\\
            Y_t &= g(X_T)-\int_t^TF(s,X_s,Y_s,Z_s)\mathrm{d}s-\int_t^TZ_s\,\mathrm{d}W_s,
        \end{aligned} \right.
    \end{align}
    given deterministic functions
    $b\colon[0,T]\times\mathbb{R}^d\times\mathbb{R}\times(\mathbb{R}^d)^*\to\mathbb{R}^d$,
    $\sigma\colon[0,T]\times\mathbb{R}^d\times\mathbb{R}\to\mathbb{R}^{d\times d}$,
    $F\colon[0,T]\times\mathbb{R}^d\times\mathbb{R}\times(\mathbb{R}^d)^*\to\mathbb{R}$,
    $g\colon\mathbb{R}^d\to\mathbb{R}$
    and an initial value $x_0\in\mathbb{R}^d$, where $T\in\mathbb{R}_{>0}$ is a fixed time horizon and $d\in\mathbb{N}$. The equation is given on a complete filtered probability space $(\Omega,\mathcal{F},(\mathcal{F}_t)_{t\geq0},\mathbb{P})$, which carries a $d$\nobreakdash-dimensional standard Brownian motion~$(W_t)_{t\geq0}$, such that $(\mathcal{F}_t)_{t\geq0}$ is the natural filtration of~$(W_t)_{t\geq0}$ augmented with all the $\mathbb{P}$-null sets. Recall that a solution of the FBSDE~\eqref{eq:FBSDE} is a triple $(X,Y,Z)\coloneqq(X_t,Y_t,Z_t)_{t\geq0}$ of $(\mathcal{F}_t)_{t\geq0}$-progressively measurable stochastic processes with values in $\mathbb{R}^d\times\mathbb{R}\times(\mathbb{R}^d)^*$, such that $X$ and $Y$ are continuous,
    $\mathbb{E}\bigl[\sup_{t\in[0,T]}\lvert X_t\rvert^2\bigr],\mathbb{E}\bigl[\sup_{t\in[0,T]}\lvert Y_t\rvert^2\bigr],\int_0^T\mathbb{E}\bigl[\lvert Z_t\rvert^2\bigr]<\infty$,
    and, for any $t\in[0,T]$, FBSDE~\eqref{eq:FBSDE} holds $\mathbb{P}$-almost surely.

    FBSDEs originate from the study of the Pontryagin maximum principle for stochastic control by \citeauthor{Bismut_1978}~\cite{Bismut_1978} and  \citeauthor{Bensoussan_1982}~\cite{Bensoussan_1982}, see also \cite{Yong_Zhou_1999}. Besides several applications to mathematical finance \cite[Chapter~8]{Ma_Yong_2007}, FBSDEs are also closely connected to quasi-linear parabolic partial differential equations (PDEs) via non-linear Feynman-Kac representations \cite{Pardoux_Peng_1992,Pardoux_Tang_1999} and the Ma-Protter-Yong four-step scheme~\cite{Ma_Protter_Yong_1994}.

    Well-posedness of FBSDEs can be established via fixed point iteration \cite{Antonelli_1993, Pardoux_Tang_1999}  either for small time horizons or under certain restrictions on the type of the coupling (weak coupling) of $X$ into $(F,g)$ or on $(Y,Z)$ into $(b,\sigma)$, or under monotonicity assumptions on the coefficients. Alternatively, under sufficient smoothness and growth conditions on the coefficients, the FBSDE can be decoupled via the solution to the associated   quasi-linear PDE, leading to well-posedness via the four-step  scheme of \citeauthor{Ma_Protter_Yong_1994}~\cite{Ma_Protter_Yong_1994}. The underlying idea of decoupling the equations finally led to the unifying approach to well-posedness of FBSDEs via decoupling fields by \citeauthor{Ma_Wu_Zhang_Zhang_2015}~\cite{Ma_Wu_Zhang_Zhang_2015}.

   The close connection between FBSDEs and parabolic PDEs can also be exploited either way for the design of numerical schemes. E.\,g., by mimicking the four-step scheme, numerical techniques for PDEs can be made viable for the approximation of FBSDEs \cite{Douglas_Ma_Protter_1996,Ma_Shen_Zhao_2008}, whereas directly approximating the FBSDEs and resorting to the Feynman-Kac representation lead to new probabilistic schemes for quasi-linear parabolic PDEs \cite{Delarue_Menozzi_2006}. For a survey on numerical methods for BSDEs and FBSDEs, we refer to \citeauthor{Chessari_Kawai_Shinozaki_Yamada_2023}~\cite{Chessari_Kawai_Shinozaki_Yamada_2023}.

   In this paper, we focus on the deep BSDE method which has been introduced by \citeauthor{E_Han_Jentzen_2017}~\cite{E_Han_Jentzen_2017,Han_Jentzen_E_2018} for decoupled FBSDEs (i.\,e., for $(b,\sigma)$ independent of $(Y,Z)$). We note that this method is closely related to the physics informed neural networks \cite{Raissi_Perdikaris_Karniadakis_2019} in PDE numerics, see \cite{Nuesken_Richter_2023} for a comparison. The key idea of the deep BSDE method is to re-write the BSDE for $Y$ in forward form as
   \begin{align*}
   	Y_t=y_0+\int_0^tF(s,X_s,Y_s,Z_s)\,\mathrm{d}s+\int_0^t Z_s\,\mathrm{d}W_s,
   \end{align*}
    and to consider the control problem to choose $(y_0,Z)$ in such a way that the terminal loss
    \begin{align*}
	    	\mathbb{E}[\lvert Y_T-g(X_T)\rvert^2]
    \end{align*}
	becomes minimal. This idea can be traced back to \citeauthor{Cvitanic_Zhang_2005}~\cite{Cvitanic_Zhang_2005} who theoretically analyze the direction of steepest descent for this minimization problem in continuous time. The deep BSDE method \cite{E_Han_Jentzen_2017,Han_Jentzen_E_2018} makes this basic idea practical, by discretizing this control problem in time, parameterizing $Z$ by a neural network, and applying stochastic gradient descent to minimize the terminal loss. In this approach, the terminal loss serves as a surrogate for the unobservable  $L^2$-error between true solution and numerical approximation. Hence, in order to justify the deep BSDE method, a-posteriori estimates for controlling the true $L^2$-error in terms of the terminal loss are required. For decoupled FBSDEs, such a-posteriori estimates had already been available in the literature \cite{Bender_Steiner_2013}. For coupled FBSDEs, the deep BSDE method has first been analyzed by \citeauthor{Han_Long_2020}~\cite{Han_Long_2020}. They consider the special case without $Z$-coupling and impose standard assumptions on a short time horizon or monotonicity or weak coupling (for brevity, we will refer to such a set of conditions as the weak coupling case from now on). \citeauthor{Han_Long_2020}~\cite{Han_Long_2020} derive suitable a-posteriori estimates for discrete-time versions of the coupled FBSDE system, which are finally combined with bounds for the time-discretization error due to \citeauthor{Bender_Zhang_2008}~\cite{Bender_Zhang_2008}. Recently, the approach of \citeauthor{Han_Long_2020}~\cite{Han_Long_2020} has been extended to the case of $Z$-coupling into the drift coefficient $b$ by \citeauthor{Negyesi_Huang_Oosterlee_2026}~\cite{Negyesi_Huang_Oosterlee_2026}; see Remark \ref{rem:weak_coupling} for further discussion. Several variants and extensions of the deep BSDE method can be found in the literature (e.\,g.,  \cite{Gao_Gao_Hu_Zhu_2023,Huang_Negyesi_Oosterlee_2025,Reisinger_Stockinger_Zhang_2024}), which typically rely on a weak coupling assumption in one way or another. A notable exception is the work by \citeauthor{Jiang_Li_2021}~\cite{Jiang_Li_2021}, who  assume sufficient smoothness  conditions  for the associated parabolic PDE to be well-posed, but consider the case in which coupling of the BSDE into the SDE is via $Y$ and through the drift coefficient $b$ only.
    \citeauthor{Andersson_Andersson_Oosterlee_2023}~\cite{Andersson_Andersson_Oosterlee_2023} provide a numerical test case with strong coupling, in which the deep BSDE method empirically fails to converge, and suggest a more robust variant of the deep BSDE method for FBSDEs arising in stochastic control. They attribute this failure of convergence to the $Z$-coupling: ``This is only possible if the coupling of $Z$ in $b$ is strong enough, so that $g(X_T)$ can be efficiently
    controlled by $Z$'' \cite[Section~2.3, p.~A232]{Andersson_Andersson_Oosterlee_2023}.

    In this paper, we prove, to the best of our knowledge, the first a-posteriori estimates that justify the use of the deep BSDE method for FBSDEs with strong $Z$-coupling through the drift coefficient (and we, additionally, allow strong $Y$- and $X$-coupling in all coefficient functions). Similarly to \citeauthor{Jiang_Li_2021}~\cite{Jiang_Li_2021} we impose regularity and growth conditions on the coefficients which guarantee existence of a sufficiently `nice' classical solution to the associated quasi-linear parabolic PDE. For the proof of our main result (Theorem~\ref{theorem:main}), we identify an auxiliary decoupled FBSDE, to which we can apply the a-posteriori estimates of \citeauthor{Bender_Steiner_2013}~\cite{Bender_Steiner_2013} and, then, analyze the error between the solutions to the auxiliary decoupled FBSDE and the original coupled FBSDE. Our results suggest that the failure of the deep BSDE method in the test case in \cite{Andersson_Andersson_Oosterlee_2023} may  rather be attributed to the combination of of growth properties of the coefficient functions with the strong $Z$-coupling. We also demonstrate our results numerically for a test case, in which the weak coupling conditions imposed in \cite{Negyesi_Huang_Oosterlee_2026} are violated.

	This paper is organized as follows. Section~\ref{sec:setting} introduces the  framework and the standing assumptions. In Section~\ref{sec:deep_bsde}, we discuss the proposed deep BSDE algorithm and state the main convergence theorem. Section~\ref{sec:numerical_experiments} illustrates the method by numerical examples. The proof of the main result is given in Section~\ref{sec:proof_main_theorem}. Section~\ref{sec:sufficient_conditions} presents easy-to-verify sufficient conditions on the coefficients ensuring the applicability of our results. 
	
	\section{General Setting and Standing Assumptions}\label{sec:setting}
	In what follows, we denote by $|\cdot|$ the Euclidean norm for (column and row) vectors and the Frobenius norm for matrices. When vectors are identified with single-column or single-row matrices, these norms coincides. Consequently, properties of the Frobenius norm, such as sub-multiplicativity, apply without further distinction to products of vectors and matrices.
	
	We impose the following assumptions on $b$, $\sigma$, $\sigma^{-1}$, $\sigma\sigma^\top$, $F$, and $g$, where we formally define the functions $\sigma^{-1}$ and $\sigma\sigma^\top$ as
	\begin{align*}
		\sigma^{-1}			&\colon[0,T]\times\mathbb{R}^d\times\mathbb{R}\to\mathbb{R}^{d\times d},\quad(t,x,y)\mapsto\sigma(t,x,y)^{-1},\\
		\sigma\sigma^\top	&\colon[0,T]\times\mathbb{R}^d\times\mathbb{R}\to\mathbb{R}^{d\times d},\quad(t,x,y)\mapsto\sigma(t,x,y)\sigma(t,x,y)^\top.
	\end{align*}
	
	\begin{assumption}\label{ass:functions}
		\begin{itemize}
			\item[(i)]
				The matrix inverse $\sigma^{-1}(t,x,y)$ of $\sigma(t,x,y)$ exists for all $(t,x,y)\in[0,T]\times\mathbb{R}^d\times\mathbb{R}$.
			\item[(ii)]
				There exist constants $H_b,H_\sigma,H_{\sigma^{-1}},H_{\sigma\sigma^\top},H_F,H_g\in\mathbb{R}_{>0}$ such that for any two $(t,x,y,z)$, $(t',x',y',z')\in[0,T]\times\mathbb{R}^d\times\mathbb{R}\times(\mathbb{R}^d)^*$ it holds that
				\begin{align*}
					\lvert b(t,x,y,z)-b(t',x',y',z')\rvert&\leq H_b\bigl(\lvert t-t'\rvert^{1/2}+\lvert x-x'\rvert+\lvert y-y'\rvert+\lvert z-z'\rvert\bigr),\\
					\lvert\sigma(t,x,y)-\sigma(t',x',y')\rvert&\leq H_\sigma\bigl(\lvert t-t'\rvert^{1/2}+\lvert x-x'\rvert+\lvert y-y'\rvert\bigr),\\
					\lvert\sigma^{-1}(t,x,y)-\sigma^{-1}(t',x',y')\rvert&\leq H_{\sigma^{-1}}\bigl(\lvert t-t'\rvert^{1/2}+\lvert x-x'\rvert+\lvert y-y'\rvert\bigr),\\
					\lvert(\sigma\sigma^\top)(t,x,y)-(\sigma\sigma^\top)(t',x',y')\rvert&\leq H_{\sigma\sigma^\top}\bigl(\lvert t-t'\rvert^{1/2}+\lvert x-x'\rvert+\lvert y-y'\rvert\bigr),\\
					\lvert F(t,x,y,z)-F(t',x',y',z')\rvert&\leq H_F\bigl(\lvert t-t'\rvert^{1/2}+\lvert x-x'\rvert+\lvert y-y'\rvert+\lvert z-z'\rvert\bigr),\\
					\lvert g(x)-g(x')\rvert&\leq H_g\lvert x-x'\rvert,
				\end{align*}
				i.\,e., the functions~$b$ and~$F$ are $\frac{1}{2}$-Hölder continuous in~$t$ and Lipschitz continuous in $(x,y,z)$ with constants~$H_b$ and~$H_F$, the functions $\sigma$, $\sigma^{-1}$ and $\sigma\sigma^\top$ are $\frac{1}{2}$-Hölder continuous in~$t$ and Lipschitz continuous in $(x,y)$ with constants~$H_\sigma$, $H_{\sigma^{-1}}$ and $H_{\sigma\sigma^\top}$, and the function~$g$ is Lipschitz continuous in~$x$ with constant~$H_g$.
			\item[(iii)]
				The functions $b$, $\sigma$, $\sigma^{-1}$, and $\sigma\sigma^\top$ are bounded by constants $K_b,K_\sigma,K_{\sigma^{-1}},K_{\sigma\sigma^\top}\in\mathbb{R}_{>0}$.
		\end{itemize}
	\end{assumption}
	
	\begin{remark}\label{rem:Lipschitz}
		\begin{itemize}
			\item[(i)]
				Since the Frobenius norm is sub-multiplicative, the boundedness of $\sigma$ already implies the boundedness of $\sigma\sigma^\top$.
			\item[(ii)]
				In connection with the boundedness of~$\sigma$, it is easily seen that the $\frac{1}{2}$-Hölder continuity in~$t$ and the Lipschitz continuity in $(x,y)$ of $\sigma$ are transferred to $\sigma\sigma^\top$. Indeed, for each two $(t,x,y),(t',x',y')\in[0,T]\times\mathbb{R}^d\times\mathbb{R}$, applying the triangle inequality and the sub-multiplicativity of the Frobenius norm, it follows that
				\begin{align*}
					&\lvert(\sigma\sigma^\top)(t,x,y)-(\sigma\sigma^\top)(t',x',y')\rvert\\
					&=\begin{multlined}[t]
						\lvert\sigma(t,x,y)\sigma(t,x,y)^\top-\sigma(t,x,y)\sigma(t',x',y')^\top\\
						+\sigma(t,x,y)\sigma(t',x',y')^\top-\sigma(t',x',y')\sigma(t',x',y')^\top\rvert
					\end{multlined}\\
					&\leq\lvert\sigma(t,x,y)\rvert\lvert\sigma(t,x,y)-\sigma(t',x',y')\rvert+\lvert\sigma(t,x,y)-\sigma(t',x',y')\rvert\lvert\sigma(t',x',y')\rvert\\
					&\leq2K_\sigma H_\sigma\bigl(\lvert t-t'\rvert^{1/2}+\lvert x-x'\rvert+\lvert y-y'\rvert\bigr).
				\end{align*}
			\item[(iii)]
				Similarly, but now relying on the boundedness of $\sigma^{-1}$, the $\frac{1}{2}$-Hölder continuity in $t$ and the Lipschitz continuity in $(x,y)$ of $\sigma$ also carry over to $\sigma^{-1}$: For each two $(t,x,y),(t',x',y')\in[0,T]\times\mathbb{R}^d\times\mathbb{R}$, it holds that
				\begin{align*}
					\lvert\sigma^{-1}(t,x,y)-\sigma^{-1}(t',x',y')\rvert
					&=\bigl\lvert\sigma^{-1}(t',x',y')\bigl(\sigma(t',x',y')-\sigma(t,x,y)\bigr)\sigma^{-1}(t,x,y)\bigr\rvert\\
					&\leq\lvert\sigma^{-1}(t',x',y')\rvert\lvert\sigma(t',x',y')-\sigma(t,x,y)\rvert\lvert\sigma^{-1}(t,x,y)\rvert\\
					&\leq K_{\sigma^{-1}}^2H_\sigma\bigl(\lvert t-t'\rvert^{1/2}+\lvert x-x'\rvert+\lvert y-y'\rvert\bigr).
				\end{align*}
		\end{itemize}
	\end{remark}
	
	\begin{remark}
		By standard results, it holds that the existence of $\sigma^{-1}$ together with the boundedness of $\sigma$ and $\sigma^{-1}$ is equivalent to the uniform ellipticity condition
		\begin{align}\label{eq:elliptic}
			\forall(t,x,y)\in[0,T]\times\mathbb{R}^d\times\mathbb{R}\colon\quad\nu I\leq\sigma(t,x,y)\sigma(t,x,y)^\top\leq\mu I,
		\end{align}
		where $\nu\leq\mu$ are some positive real constants and $I=\mathrm{diag}(1,\ldots,1)\in\mathbb{R}^{d\times d}$ denotes the identity matrix. Thus, together with Remark~\ref{rem:Lipschitz}, it follows that Assumption~\ref{ass:functions} can be reduced to the Hölder and Lipschitz conditions on $b$, $\sigma$, $F$, and $g$, the boundedness condition on $b$, and the above uniform ellipticity condition.
	\end{remark}
	
	With the following additional assumption we ensure that FBSDE~\eqref{eq:FBSDE} is uniquely solvable.
	\begin{assumption}\label{ass:sol_PDE}
			\begin{itemize}
			\item[(a)]
				The quasi-linear parabolic PDE
				\begin{align}\label{eq:PDE}
					\left\{\begin{aligned}
						u_t(t,x)&=F(t,x,u(t,x),v(t,x))-\frac{1}{2}\mathrm{tr}\bigl[u_{xx}(t,x)(\sigma\sigma^\top)(t,x,u(t,x))\bigr]\\
						&\quad-u_x(t,x)b(t,x,u(t,x),v(t,x)),\quad(t,x)\in[0,T]\times\mathbb{R}^d,\\
						u(T,x)&=g(x),\quad x\in\mathbb{R}^d,
					\end{aligned}\right.
				\end{align}
				where $v\colon[0,T]\times\mathbb{R}^d\to(\mathbb{R}^d)^*$ is defined via
				\begin{align*}
					v(t,x)\coloneqq u_x(t,x)\sigma(t,x,u(t,x)),
				\end{align*}
				admits a classical solution $u\colon[0,T]\times\mathbb{R}^d\to\mathbb{R}$ such that its derivatives $u_x$ and $u_{xx}$ are bounded by constants $K_{u_x}\in\mathbb{R}_{>0}$ and $K_{u_{xx}}\in\mathbb{R}_{>0}$.
				Here, considering $x=(x_1,\ldots,x_d)^\top$, the derivatives $u_t$, $u_x$, and $u_{xx}$ of $u$ are given as
				\begin{align*}
					u_t(t,x)&\coloneqq\frac{\partial u(t,x)}{\partial t}\in\mathbb{R},\\
					u_x(t,x)&\coloneqq(u_{x_i})_{i=1,\ldots,d}\coloneqq\left(\frac{\partial u(t,x)}{\partial x_i}\right)_{i=1,\ldots,d}\in(\mathbb{R}^d)^*,\\
					u_{xx}(t,x)&\coloneqq(u_{x_ix_j}(t,x))_{i,j=1,\ldots,d}\coloneqq\left(\frac{\partial^2 u(t,x)}{\partial x_i\partial x_j}\right)_{i,j=1,\ldots,d}\in\mathbb{R}^{d\times d}.
				\end{align*}
			\item[(b)]
				There exist constants $H_u,H_{u_x},H_{u_{xx}}\in\mathbb{R}_{>0}$ and an $\alpha\in(0,1]$
		 such that the functions $u$, $u_x$ and $u_{xx}$ from (a) fulfill the following Hölder and Lipschitz conditions: For all $(t,x),(t',x')\in[0,T]\times\mathbb{R}^d$ it holds that
				\begin{align*}
					\lvert u(t,x)-u(t',x')\rvert&\leq H_u\bigl(\lvert t-t'\rvert^{1/2}+\lvert x-x'\rvert\bigr),\\
					\lvert u_x(t,x)-u_x(t',x')\rvert&\leq H_{u_x}\bigl(\lvert t-t'\rvert^{1/2}+\lvert x-x'\rvert\bigr),\\
					\lvert u_{xx}(t,x)-u_{xx}(t',x')\rvert&\leq H_{u_{xx}}\bigl(\lvert t-t'\rvert^{\alpha/2}+\lvert x-x'\rvert^\alpha\bigr),
				\end{align*}
				i.\,e., $u$ and $u_x$ are $\frac{1}{2}$-Hölder continuous in~$t$ and Lipschitz continuous in~$x$ with constants~$H_u$ and $H_{u_x}$ and~$u_{xx}$ is $\frac{\alpha}{2}$-Hölder continuous in~$t$ and $\alpha$-Hölder continuous in~$x$ with constant~$H_{u_{xx}}$.
		\end{itemize}
	\end{assumption}
	
	\begin{remark}
		\begin{itemize}
			\item[(a)] Observe that $u_{xx}(t,x)=u_{xx}(t,x)^\top$ since for a classical solution $u$ of PDE~\eqref{eq:PDE} the derivative~$u_{xx}$ is assumed to be continuous.
			\item[(b)] The Lipschitz continuity of~$u$ and~$u_x$ in~$x$ already follows by the boundedness of~$u_x$ and~$u_{xx}$, respectively.
		\end{itemize}
	\end{remark}
	
	Assumptions~\ref{ass:functions} and~\ref{ass:sol_PDE} will be in force throughout the whole paper without further notice.
	
	\begin{lem}\label{lem:v}
		There exist (minimal) constants $H_v,K_v\in\mathbb{R}_{\geq0}$ with
		\begin{align*}
			H_v\leq K_\sigma H_{u_x}+K_{u_x}H_\sigma(1+H_u)
			\quad\text{and}\quad
			K_v\leq K_{u_x}K_\sigma
		\end{align*}
		such that
		\begin{itemize}
			\item[(a)]
				for all $(t,x),(t',x')\in[0,T]\times\mathbb{R}^d$ it holds that
				\begin{align*}
					\lvert v(t,x)-v(t',x')\rvert&\leq H_v\bigl(\lvert t-t'\rvert^{1/2}+\lvert x-x'\rvert\bigr),
				\end{align*}
				i.\,e., the function $v$ is $\frac{1}{2}$-Hölder continuous in~$t$ and Lipschitz continuous in~$x$ with common constant $H_v$, and 
			\item[(b)]
				$v$ is bounded by $K_v$.
		\end{itemize}
	\end{lem}
	\begin{proof}
		Let $(t,x),(t',x')\in[0,T]\times\mathbb{R}^d$. The Hölder/Lipschitz condition follows from
		\begin{align*}
			&\lvert u_x(t,x)\sigma(t,x,u(t,x))-u_x(t',x')\sigma(t',x',u(t',x'))\rvert\\
			&\leq\begin{multlined}[t]
				\lvert u_x(t,x)\sigma(t,x,u(t,x))-u_x(t',x')\sigma(t,x,u(t,x))\rvert\\
				+\lvert u_x(t',x')\sigma(t,x,u(t,x))-u_x(t',x')\sigma(t',x',u(t',x'))\rvert
			\end{multlined}\\
			&\leq(K_\sigma H_{u_x}+K_{u_x}H_\sigma(1+H_u))\bigl(\lvert t-t'\rvert^{1/2}+\lvert x-x'\rvert\bigr).
		\end{align*}
		The boundedness condition follows directly by the sub-multiplicativity of the Frobenius norm from the boundedness of $u_x$ and $\sigma$.
	\end{proof}
	
	\begin{remark}\label{rem:HLC_uxsigma}
		Define
		\begin{align*}
			\vartheta\colon[0,T]\times\mathbb{R}^d\times\mathbb{R}\to(\mathbb{R}^d)^*
			\quad
			(t,x,y)\mapsto u_x(t,x)\sigma(t,x,y)
		\end{align*}
		Similarly to Lemma~\ref{lem:v}, there exist constants $H_\vartheta,K_\vartheta\in\mathbb{R}_{>0}$ with
		\begin{align*}
			H_\vartheta\leq K_{u_x}H_\sigma+K_\sigma H_{u_x}
			\quad\text{and}\quad
			K_\vartheta\leq K_{u_x}K_\sigma
		\end{align*}
		such that for all $(t,x,y),(t',x',y')\in[0,T]\times\mathbb{R}^d\times\mathbb{R}$ it holds that
		\begin{align*}
			\lvert\vartheta(t,x,y)-\vartheta(t',x',y')\rvert
			\leq H_\vartheta\bigl(\lvert t-t'\rvert^{1/2}+\lvert x-x'\rvert+\lvert y-y'\rvert\bigr)
		\end{align*}
		and $\vartheta$ is bounded by $K_\vartheta$.
	\end{remark}
	
	\begin{theorem}
		FBSDE~\eqref{eq:FBSDE} admits a unique adapted solution $(X,Y,Z)$. Especially, the processes~$Y$ and~$Z$ take the following form: For any $t\in[0,T]$ it holds that
		\begin{align*}
			Y_t=u(t,X_t)
			\quad\text{and}\quad
			Z_t=v(t,X_t).
		\end{align*}
	\end{theorem}
	\begin{proof}
		Follows from \cite[Chapter~4, Theorem~1.1]{Ma_Yong_2007}. Revisiting the proof of \cite[Chapter~4, Theorem~1.1]{Ma_Yong_2007}, we observe that we do not need the uniform Lipschitz continuity in $(x,y,z)$ of the mapping $(t,x,y,z)\mapsto z\sigma(t,x,y)$ but the uniform Lip\-schitz continuity of $v$ in $x$ (see Lemma~\ref{lem:v}) is already sufficient. 
	\end{proof}
	
	\begin{remark}
		To see that $Y_t=u(t,X_t)$ and $Z_t=v(t,X_t)=u_x(t,X_t)\sigma(t,X_t,u(t,X_t))$ indeed solve the backward stochastic differential equation for $Y_t$ of FBSDE~\eqref{eq:FBSDE}, it suffices to apply Itô's formula to $u(t,X_t)$ and insert the identities for $u_t(s,X_s)$ and $u(T,X_T)$ given by the quasi-linear parabolic PDE~\eqref{eq:PDE}:
		\begin{align*}
			Y_t
			&=\begin{multlined}[t]
				u(T,X_T)-\int_t^Tu_t(s,X_s)+u_x(s,X_s)b(s,X_s,Y_s,Z_s)\\
				+\frac{1}{2}\mathrm{tr}\bigl[u_{xx}(s,X_s)(\sigma\sigma^\top)(s,X_s,Y_s)\bigr]\,\mathrm{d}s-\int_t^Tu_x(s,X_s)\sigma(s,X_s,Y_s)\,\mathrm{d}W_s.
			\end{multlined}\\
			&=\begin{multlined}[t]
				g(X_T)-\int_t^TF(s,X_s,u(s,X_s),v(s,X_s))-u_x(s,X_s)b(s,X_s,u(s,X_s),v(s,X_s))\\
				-\frac{1}{2}\mathrm{tr}\bigl[u_{xx}(s,X_s)(\sigma\sigma^\top)(s,X_s,u(s,X_s))\bigr]+u_x(s,X_s)b(s,X_s,Y_s,Z_s)\\+\frac{1}{2}\mathrm{tr}\bigl[u_{xx}(s,X_s)(\sigma\sigma^\top)(s,X_s,Y_s)\bigr]\,\mathrm{d}s-\int_t^Tu_x(s,X_s)\sigma(s,X_s,u(s,X_s))\,\mathrm{d}W_s
			\end{multlined}\\
			&=g(X_T)-\int_t^TF(s,X_s,Y_s,Z_s)\,\mathrm{d}s-\int_t^TZ_s\,\mathrm{d}W_s.
		\end{align*}	
	\end{remark}
	
	\section{Approximation via a Deep BSDE Method}\label{sec:deep_bsde}
	Let $\pi\coloneqq\{t_0,\ldots,t_N\}\subseteq[0,T]$ be a time grid with $0\eqqcolon t_0<\ldots<t_N\coloneqq T$. Define the function $\Pi\colon [0,T]\to\pi$ by $\Pi(t)\coloneqq t_i$ for $t\in[t_i,t_{i+1})$, $i=0,\ldots,N-1$, and $\Pi(T)\coloneqq t_N=T$. Assume for a moment that we already know the initial condition $Y_0=u(0,x_0)$ of the process~$Y$ and the functional relationship $Z_t=v(t,X_t)$ while the functional relationship $Y_t=u(t,X_t)$ is still unknown. In this situation, one possible strategy to approximate the processes~$X$ and~$Y$ would be to reformulate FBSDE~\eqref{eq:FBSDE} as a system of forward SDEs only involving the processes~$X$ and~$Y$ and to apply the Euler-Maruyama method. Back in reality, where the necessary information is not given, we are only able to mimic the proposed strategy by instead searching for a tuple $\eta\coloneqq(y_0,\rho)$ of an initial value $y_0\in\mathbb{R}$ and a candidate function $\rho\colon[0,T]\times\mathbb{R}^d\to(\mathbb{R}^d)^*$ such that the initial condition $Y_0\approx y_0$ and the functional relationship $Z_t\approx\rho(t,X_t)$ hold approximately. Then, on $[0,T]$, we approximate the adapted solution $(X,Y,Z)$ by the triple $(X^{\eta,\pi},Y^{\eta,\pi},Z^{\eta,\pi})\coloneqq(X^{\eta,\pi}_t,Y^{\eta,\pi}_t,Z^{\eta,\pi}_t)_{t\in[0,T]}$, which is determined via
	\begin{align}\label{eq:Euler}
		\left\{\begin{aligned} 
			X^{\eta,\pi}_{t_0}&\coloneqq x_0,\;Y^{\eta,\pi}_{t_0}\coloneqq y_0,\\
			Z^{\eta,\pi}_{t_i}&\coloneqq\rho(t_i,X^{\eta,\pi}_{t_i}),\quad\forall i\in\{0,\ldots,N\},\\			
			X^{\eta,\pi}_{t_{i+1}}&\coloneqq X^{\eta,\pi}_{t_i}+b(t_i,X^{\eta,\pi}_{t_i},Y^{\eta,\pi}_{t_i},Z^{\eta,\pi}_{t_i})\Delta_i+\sigma(t_i,X^{\eta,\pi}_{t_i},Y^{\eta,\pi}_{t_i})\Delta W_i,\quad\forall i\in\{0,\ldots,N-1\},\\
			Y^{\eta,\pi}_{t_{i+1}}&\coloneqq Y^{\eta,\pi}_{t_i}+F(t_i,X^{\eta,\pi}_{t_i},Y^{\eta,\pi}_{t_i},Z^{\eta,\pi}_{t_i})\Delta_i+Z^{\eta,\pi}_{t_i}\Delta W_i,\quad\forall i\in\{0,\ldots,N-1\},\\
			X^{\eta,\pi}_t&\coloneqq X^{\eta,\pi}_{\Pi(t)},\;Y^{\eta,\pi}_t\coloneqq Y^{\eta,\pi}_{\Pi(t)},\;Z^{\eta,\pi}_t\coloneqq Z^{\eta,\pi}_{\Pi(t)},\quad\forall t\in[0,T]\setminus\pi,
		\end{aligned}\hspace{-0.29cm}\right.
	\end{align}
	where, for $i\in\{0,\ldots,N-1\}$, $\Delta_i\coloneqq t_{i+1}-t_i\in\mathbb{R}_{>0}$ and $\Delta W_i\coloneqq W_{t_{i+1}}-W_{t_i}\in\mathbb{R}^{d}$.
	
	To evaluate the quality of the approximating solution triple $(X^{\eta,\pi},Y^{\eta,\pi},Z^{\eta,\pi})$, i.\,e., to evaluate how close it is to the true adapted solution $(X,Y,Z)$ on $[0,T]$, we are interested in the following $L^2$-approximation error:
	\begin{align*}
		\mathcal{E}_{\eta,\pi}\coloneqq\sup_{t\in[0,T]}\mathbb{E}\bigl[\lvert X_t-X^{\eta,\pi}_t\rvert^2\bigr]+\max_{t_i\in\pi}\mathbb{E}\bigl[\lvert Y_{t_i}-Y^{\eta,\pi}_{t_i}\rvert^2\bigr]+\mathbb{E}\biggl[\int_0^T\lvert Y_t-Y^{\eta,\pi}_t\rvert^2+\lvert Z_t-Z^{\eta,\pi}_t\rvert^2\,\mathrm{d}t\biggr].
	\end{align*}
	
	We consider the following function class:
	\begin{align*}
		\Theta_{K_\rho,H_\rho}\coloneqq\left\{\rho\colon[0,T]\times\mathbb{R}^d\to(\mathbb{R}^d)^*\;\left\vert\;
			\begin{aligned}
				&\rho\text{ is bounded by }K_\rho\in\mathbb{R}_{>0}\text{, }\rho\text{ is }\tfrac{1}{2}\text{-Hölder continuous}\\
				&\text{in the first variable and Lipschitz continuous in the}\\
				&\text{second variable with common constant }H_\rho\in\mathbb{R}_{>0}\text{,}\\
				&\text{i.\,e., for all }(t,x),(t',x')\in[0,T]\times\mathbb{R}^d\text{ it holds that}\\
				&\quad\lvert\rho(t,x)-\rho(t',x')\rvert\leq H_\rho\bigl(\lvert t-t'\rvert^{1/2}+\lvert x-x'\rvert\bigr).
			\end{aligned}
			\right.\right\}.
	\end{align*}
	It turns out that, if choosing $\rho\in\Theta_{K_\rho,H_\rho}$, the $L^2$-approximation error $\mathcal{E}_{\eta,\pi}$ can be bounded solely considering the expected squared difference $\mathbb{E}\bigl[\lvert Y^{\eta,\pi}_{T_N}-g(X^{\eta,\pi}_{T_N})\rvert^2\bigr]$, which measures how good the approximating processes $X^{\eta,\pi}_{T_N}$ and $Y^{\eta,\pi}_{T_N}$ fulfill the terminal condition, and the mesh size $\lvert\pi\rvert\coloneqq\max_{i=0,\ldots,N-1}\Delta_i$.
	 
	\begin{theorem}\label{theorem:main}
		Let $K_\rho,H_\rho\in\mathbb{R}_{>0}$. There exists some constant $C\in[1,\infty)$ such that, for every $\rho\in\Theta_{K_\rho,H_\rho}$, $y_0\in\mathbb{R}$ and $\lvert\pi\rvert\leq(C(1+K_\rho^4))^{-1}$, it holds that
		\begin{align*}
			\mathcal{E}_{\eta,\pi}\leq C\Bigl(\mathrm{e}^{CK_\rho^4}\mathbb{E}\bigl[\lvert Y^{\eta,\pi}_T-g(X^{\eta,\pi}_T)\rvert^2\bigr]+\mathrm{e}^{C(K_\rho^2+H_\rho^2)}(\lvert\pi\rvert+\lvert\pi\rvert^\alpha)\Bigr).
		\end{align*}
	\end{theorem}
	
	This theorem shows that, on $[0,T]$, the adapted solution $(X,Y,Z)$ of FBSDE~\eqref{eq:FBSDE} can be approximated arbitrarily well by $(X^{\eta,\pi},Y^{\eta,\pi},Z^{\eta,\pi})$, provided that the candidate function $\rho$ and the initial condition $y_0$ are chosen such that the terminal condition is satisfied in an $L^2$-optimal sense, i.\,e., $\mathbb{E}\bigl[\lvert Y^{\eta,\pi}_T-g(X^{\eta,\pi}_T)\rvert^2\bigr]$ is minimized over all $\eta\in\mathbb{R}\times\Theta_{K_\rho,H_\rho}$.  Here, in view of Lemma~\ref{lem:v}, it is natural to choose $K_\rho\geq K_v$ and $H_\rho\geq H_v$, while avoiding unnecessarily large values of these constants. Otherwise, a good approximation may not be found, as the search is restricted to an inappropriate class of functions.

	\begin{remark}
		\label{rem:weak_coupling}
	\citeauthor{Han_Long_2020}~\cite{Han_Long_2020} were the first to prove a-posteriori bounds of this type for coupled FBDEs. In their setting, the drift coefficient $b$ does not depend on $Z$ (i.\,e., no $Z$-coupling) and one of the following generic assumptions must be satisfied: (i) small time horizon ($T$ small), (ii) weak coupling of $Y$ into the forward SDE (Lipschitz constants of $b$ and $\sigma$ in the $y$-variable small), (iii) weak coupling of $X$ into the backward SDE (Lipschitz constants of $F$ and $g$ in the $x$-variable small), (iv) $(-F)$ strongly decreasing in the $y$-variable, (v) $b$ strongly decreasing in the $x$-variable; see  \cite[Theorem 1']{Han_Long_2020} for a precise statement. The results of \cite{Han_Long_2020} have been generalized by \citeauthor{Negyesi_Huang_Oosterlee_2026}~\cite{Negyesi_Huang_Oosterlee_2026} to the case that $Z$ couples into the forward SDE through the drift coefficient $b$. They additionally require a weak coupling condition of $Z$ (Lipschitz constant of $b$ in the $z$-variable small) in each of the cases (ii)--(v). We note that, by means of \cite[Theorem~3.3]{Pardoux_Tang_1999}, the weak coupling condition of $Z$ is not required for the well-posedness of the FBSDE in cases (iv)--(v).
	\end{remark}

	\begin{remark}
		\begin{itemize}
			\item[(a)]
				Since we assumed $g$ to be Lipschitz continuous with constant $H_g$, as long as FBSDE~\eqref{eq:FBSDE} is solvable, the following bound applies:
				\begin{align*}
					\mathbb{E}\bigl[\lvert Y^{\eta,\pi}_T-g(X^{\eta,\pi}_T)\rvert^2\bigr]
					&=\mathbb{E}\bigl[\lvert Y^{\eta,\pi}_T-Y_T+g(X_T)-g(X^{\eta,\pi}_T)\rvert^2\bigr]\\
					&\leq2\Bigl(\mathbb{E}\bigl[\lvert Y^{\eta,\pi}_T-Y_T\rvert^2\bigr]+H_g^2\mathbb{E}\bigl[\lvert X_T-X^{\eta,\pi}_T\rvert^2\bigr]\Bigr)\\
					&\leq2(1\vee H_g^2)\mathcal{E}_{\eta,\pi}.\qedhere
				\end{align*}
				Hence, a small terminal loss $\mathcal{E}_{\eta,\pi}$ also is necessary for a good $L^2$-approximation of the true solution.
			\item[(b)]
				In the case all constants from Assumptions~\ref{ass:functions} and~\ref{ass:sol_PDE} are explicitly known together with $K_\rho$ and $H_\rho$, Equation~\eqref{eq:main_concrete} in the proof of Theorem~\ref{theorem:main} gives an explicit bound for $\mathcal{E}_{\eta,\pi}$ in terms of $\mathbb{E}\bigl[\lvert Y^{\eta,\pi}_T-g(X^{\eta,\pi}_T)\rvert^2\bigr]$, $\lvert\pi\rvert$ and $\lvert\pi\rvert^\alpha$ whenever $\lvert\pi\rvert$ fulfills the condition of Theorem~\ref{theorem:estimate_rho_approx}.
		\end{itemize}
	\end{remark}
	
	Based on Theorem~\ref{theorem:main}, we propose a deep BSDE method to approximate the solution of \mbox{FBSDE}~\eqref{eq:FBSDE}. The solution is approximated by the Euler-Maruyama scheme~\eqref{eq:Euler} on the equidistant time grid $\hat\pi_N\coloneqq\{t_0,\ldots,t_N\}\subseteq[0,T]$ defined by $t_i\coloneqq\frac{i}{N}T$, $i=0,\ldots,N$, where the tuple $\eta=(y_0,\rho)$ of the initial value~$y_0$ and the control function~$\rho$ is learned jointly by a neural network architecture $\hat\eta\coloneqq(\hat y_0,\hat\rho)$ such that the terminal mismatch is minimized. More precisely, the initial value~$y_0$ is treated as trainable parameter~$\hat{y}_0\in\mathbb{R}$, while the function~$\rho$ is parameterized by a (residual) neural subnet $\hat\rho\colon\mathbb{R}\times\mathbb{R}^d\to(\mathbb{R}^d)^*$. Let
	\begin{align*}
		\hat\Gamma\coloneqq\mathbb{R}^{w\times w}\times\mathbb{R}^{w\times w}\times\mathbb{R}^{w\times(1+d)}\times\mathbb{R}^{d\times w}\times\mathbb{R}^w\times\mathbb{R}^w\times\mathbb{R}^w\times\mathbb{R}^d
	\end{align*}
	be the possible parameter space of $\hat\rho$ and
	\begin{align*}
		\hat\gamma\coloneqq(W^r_1,W^r_2,W^i,W^o,b^r_1,b^r_2,b^i,b^o)\in\hat\Gamma
	\end{align*}
	a concrete parameter combination.
	Given fixed values $w\in\mathbb{N}$ and $K_{\hat\rho}\in\mathbb{R}_{>0}$, the neural subnet~$\hat\rho$ is defined as
	\begin{align}\label{eq:rho_hat}
		\hat\rho\colon\mathbb{R}\times\mathbb{R}^d\to\mathbb{R}^d,\quad
		(t,x)\mapsto h_{K_{\hat\rho}}\left(W^or\left(\sigma\left(W^i\begin{pmatrix}t\\x\end{pmatrix}+b^i\right)\right)+b^o\right),
	\end{align}
	with residual block
	\begin{align*}
		r&\colon\mathbb{R}^w\to\mathbb{R}^w,\quad
		s\mapsto s+W^r_2\sigma(W^r_1\phi(s)+b^r_1)+b^r_2,
	\end{align*}
	hyperbolic tangent activation function
	\begin{align*}
		\sigma\colon\mathbb{R}^w\to\mathbb{R}^w,\;
		(s_i)_{i=1,\ldots,w}\mapsto(\tanh(s_i))_{i=1,\ldots,w},
	\end{align*}
	layer normalization
	\begin{align*}
 		\phi\colon\mathbb{R}^w\to\mathbb{R}^w,\;
		(s_i)_{i=1,\ldots,w}\mapsto\left(\frac{s_i-\frac{1}{w}\sum_{j=1}^ws_j}{\sqrt{\frac{1}{w}\sum_{j=1}^w(s_j-\frac{1}{w}\sum_{k=1}^ws_k)^2+0.001}}\right)_{i=1,\ldots,w},
	\end{align*}
	and final projection
	\begin{align*}
		h_{K_{\hat\rho}}:\mathbb{R}^d\to\mathbb{R}^d,\quad
		s\mapsto
		\begin{cases}
			s & \text{if }\lvert s\rvert\leq K_{\hat\rho}\\
			K_{\hat\rho}\frac{s}{\lvert s\rvert} & \text{if }\lvert s\rvert>K_{\hat\rho}.
		\end{cases}
	\end{align*}
	onto the closed ball with radius $K_{\hat\rho}$. Thus, $\hat\rho$ is chosen from the space of functions
	\begin{align*}
		\hat\Theta_{K_{\hat\rho}}\coloneqq\bigl\{\hat\rho\colon\mathbb{R}\times\mathbb{R}^d\to\mathbb{R}^d\mid\hat\rho\text{ is of the form~\eqref{eq:rho_hat} with }\hat\gamma\in\hat\Gamma\bigr\}.
	\end{align*}
	
	Since all network layers of the above defined neural subnets are Lipschitz continuous and the projection operator preserves Lipschitz continuity, any $\hat\rho\in\hat\Theta_{K_{\hat\rho}}$ is automatically Lipschitz and consequently, since we ensure the boundedness by the final projection $h_{K_{\hat\rho}}$, also satisfies the required Hölder and Lipschitz regularity condition. In particular:
	\begin{align*}
		\forall\hat\rho\in\hat\Theta_{K_{\hat\rho}}\colon\exists H_{\hat\rho}\in\mathbb{R}_{>0}\colon\hat\rho\in\Theta_{K_{\hat\rho},H_{\hat\rho}},
	\end{align*}
	i.\,e., any candidate control $\hat\rho\in\hat\Theta_{K_{\hat\rho}}$ fulfills the assumption of Theorem~\ref{theorem:main}.
	
	Theoretically, in view of Theorem~\ref{theorem:main}, we want to find a joint network $\hat\eta\coloneqq(\hat y_0,\hat\rho)$ with parameters $(\hat y_0,\hat\gamma)$ which solves the following minimization problem:
	\begin{align}\label{eq:minimization_problem}
		\min_{\hat\eta\in\mathbb{R}\times\hat\Theta_{K_{\hat\rho}}}\mathbb{E}\bigl[\lvert Y^{\hat\eta,\hat\pi_N}_T-g(X^{\hat\eta,\hat\pi_N}_T)\rvert^2\bigr].
	\end{align}
	In practice, we approximate this expectation by simulation: Given a set of $K\in\mathbb{N}$ realizations $\{\mathbf{\Delta}w^k\coloneqq(\Delta w^k_0,\ldots,\Delta w^k_{N-1})\}_{k=1,\ldots,K}$ of the Brownian increments $(\Delta W_0,\ldots,\Delta W_{N-1})$ under $\hat\pi_N$ and taking into account the corresponding $K$ trajectories $\{(x^{\hat\eta,\hat\pi_N,k},y^{\hat\eta,\hat\pi_N,k},z^{\hat\eta,\hat\pi_N,k})\}_{k=1,\ldots,K}$ of the process $(X^{\hat\eta,\hat\pi_N},Y^{\hat\eta,\hat\pi_N},Z^{\hat\eta,\hat\pi_N})$, we define the loss function
	\begin{align*}
		\ell(\hat y_0,\hat\gamma;\{\mathbf{\Delta}w^k\}_{k=1,\ldots,K})\coloneqq\frac{1}{K}\sum_{k=1}^K\bigl\lvert y^{\hat\eta,\hat\pi_N,k}_{t_N}-g(x^{\hat\eta,\hat\pi_N,k}_{t_N})\bigr\rvert^2.
	\end{align*}
	
	Then the training of the neural net $\hat\eta$ works as follows. Iteratively, starting from an initial neural net $\hat\eta^0\coloneqq(\hat y_0^0,\hat\rho^0)$ with parameters $(\hat y_0^0,\hat\gamma^0)\in\mathbb{R}\times\hat\Gamma$, the parameters are updated according to the standard Adam optimizer (with $\beta_1=0.9$, $\beta_2=0.999$, $\varepsilon=10^{-7}$, $m_0=v_0=0$) \cite{Kingma_Ba_2017}. In particular, for each $l=1,\ldots,L$ (where $L\in\mathbb{N}$), we first simulate $K$ realizations $\{\mathbf{\Delta}w^{l,k}\coloneqq(\Delta w^{l,k}_0,\ldots,\Delta w^{l,k}_{N-1})\}_{k=1,\ldots,K}$ of the Brownian increments $(\Delta W_0,\ldots,\Delta W_{N-1})$. Given the neural net $\hat\eta^{l-1}\coloneqq(\hat y_0^{l-1},\hat\rho^{l-1})$ with parameters $(\hat y_0^{l-1},\hat\gamma^{l-1})$, we then compute the corresponding $K$ trajectories $\{(x^{\hat\eta^{l-1},\hat\pi_N,k},Y^{\hat\eta^{l-1},\hat\pi_N,k},Z^{\hat\eta^{l-1},\hat\pi_N,k})\}_{k=1,\ldots,K}$ of the process $(X^{\hat\eta^{l-1},\hat\pi_N},Y^{\hat\eta^{l-1},\hat\pi_N},Z^{\hat\eta^{l-1},\hat\pi_N})$. Finally, we update the parameters according to
	\begin{align*}
		(\hat y_0^l,\hat\gamma^l)\coloneqq(\hat y_0^{l-1},\hat\gamma^{l-1})-\lambda(l)\frac{\hat m_l}{\sqrt{\hat v_l}+\varepsilon},
	\end{align*}
	where
	\begin{align*}
		\hat m_l\coloneqq\frac{m_l}{1-\beta_1^l}
		\quad\text{and}\quad
		\hat v_l\coloneqq\frac{v_l}{1-\beta_2^l},
	\end{align*}
	with
	\begin{align*}
		m_l&\coloneqq\beta_1m_{l-1}+(1-\beta_1)G_l,\\
		v_l&\coloneqq\beta_2v_{l-1}+(1-\beta_2)G_l^2,\\
		G_l&\coloneqq\nabla_{(\hat y_0,\hat\gamma)}\ell(\hat y_0^{l-1},\hat\gamma^{l-1};\{\mathbf{\Delta}w^{l,k}\}_{k=1,\ldots,K}),
	\end{align*}
	and, for initial learning rate $\lambda_0\in\mathbb{R}_{>0}$, decay rate $k_\lambda\in\mathbb{R}_{>0}$ and decay steps $n_\lambda$,
	\begin{align*}
		\lambda\colon\{1,\ldots,L\}\to\mathbb{R}_{>0},\;l\mapsto\lambda_0 k_\lambda^{l/n_\lambda}
	\end{align*}
	is a learning rate schedule. We denote the minimizing neural net of minimization problem~\eqref{eq:minimization_problem} with
	\begin{align*}
		\hat\eta^*\coloneqq(\hat y_0^*,\hat\rho^*)\coloneqq(\hat y_0^L,\hat\rho^L),
	\end{align*}
	which has parameters $(\hat y_0^*,\hat\gamma^*)\coloneqq(\hat y_0^L,\hat\gamma^L)$ and depends on $w$, $K_{\hat\rho}$, $N$, $K$, $L$, $\hat y_0^0$, $\hat\gamma^0$, $\lambda_0$, $k_\lambda$, $n_\lambda$ and $\{\mathbf{\Delta}w^{l,k}\}_{l=1,\ldots,L;\,k=1,\ldots,K}$.
	
	\begin{remark}
		\begin{itemize}
			\item[(a)]
				While the described algorithm controls the bound~$K_{\hat\rho^*}$ of the optimizer $\hat\rho^*$, it does not explicitly enforce any bound on its Lipschitz constant~$H_{\hat\rho^*}$. Consequently, a sequence $(H_{\hat\rho^*_N})$ of Lipschitz constants corresponding to a sequence of learned minimizers $(\hat\rho^*_N)$ may diverge as $N\to\infty$. Depending on the growth rate of the sequence $(H_{\hat\rho^*_N})$, this may affect convergence as the time discretization is refined. Nevertheless, such behavior was not observed in our numerical experiments.
			\item[(b)]
				Simultaneous control of both the bound $K_{\hat\rho^*_N}$ and Lipschitz constant $H_{\hat\rho^*_N}$ of a minimizer~$\hat\rho^*_N$ would in particular be beneficial in situations where the constants $K_v$ and $H_v$ are unknown. In this case, in order for the neural network to approximate plausible candidate controls, the admissible bounds $K_{\hat\rho^*_N}$ and $H_{\hat\rho^*_N}$ eventually need to exceed $K_v$ and $H_v$. However, Theorem~\ref{theorem:main} shows that these constants cannot be allowed to grow arbitrarily fast as $N\to\infty$. Hence, to guarantee convergence of the deep BSDE method, their growth must be controlled appropriately. One possibility to additionally control the Lipschitz constant of~$\hat\rho^*_N$ is, noticing that~$\hat\rho$ is even differentiable, to penalize $\sup_{(t,x)\in[0,T]\times\mathbb{R}^d}\lvert J_{\hat\rho}(t,x)\rvert_2$ exceeding~$H_\rho$, i.\,e., for some $c\in\mathbb{R}_{\geq0}$, solving the minimization problem
				\begin{align*}
					\min_{\hat\eta\in\mathbb{R}\times\hat\Theta_{K_{\hat\rho}}}\mathbb{E}\bigl[\lvert Y^{\hat\eta,\hat\pi_N}_T-g(X^{\hat\eta,\hat\pi_N}_T)\rvert^2\bigr]+c\left(\sup_{(t,x)\in[0,T]\times\mathbb{R}^d}\lvert J_{\hat\rho}(t,x)\rvert_2-H_\rho\right)^+,
				\end{align*}
				where $J_{\hat\rho}(t,x)$ denotes the Jacobian of $\hat\rho$ evaluated at $(t,x)$, $\lvert\cdot\rvert_2$ denotes the spectral norm and $(\cdot)^+\coloneqq\max\{\cdot,0\}$.
				In the learning process of the neural net $(\hat y_0,\hat\rho)$, we could adjust the loss function accordingly as
				\begin{align*}
					&\ell(\hat y_0,\hat\gamma;\{\mathbf{\Delta}w^k\}_{k=1,\ldots,K})\\
					&\quad\coloneqq\frac{1}{K}\sum_{k=1}^K\bigl\lvert y^{\hat\eta,\hat\pi_N,k}_{t_N}-g(x^{\hat\eta,\hat\pi_N,k}_{t_N})\bigr\rvert^2+\frac{c}{\gamma}\log\left(\sum_{k=1}^K\sum_{i=0}^N\mathrm{e}^{\gamma\bigl(\lvert J_{\hat\rho}(t_i,\hat X^{l,k}_{t_i})\rvert_2-H_\rho\bigr)^+}\right),
				\end{align*}
				where $\gamma\in[1,\infty)$. This approximates the maximal violation of the Hölder constant.
			\item[(c)]
				While in our approach~$\rho$ is approximated by a single neural network~$\hat\rho$ which also takes~$t$ as an input, the original deep BSDE approach uses, for each $i=0,\ldots,N$, a separate neural network~$\hat\rho_i$ to approximate $\rho(t_i,\cdot)$. However, the use of separate neural networks does not directly guarantee that the resulting approximation~$\hat\rho$ satisfies the Hölder and Lipschitz conditions required in our framework. Moreover, deep BSDE methods in the literature typically rely on standard feed-forward neural networks rather than residual neural networks.
		\end{itemize}
	\end{remark}
	
	\section{Numerical Experiments}\label{sec:numerical_experiments}
	\begin{example}\label{ex}
		The following example is a modified version of \cite[Example~1]{Negyesi_Huang_Oosterlee_2026} such that~$b$ and~$\sigma$ fulfill the necessary boundedness conditions of Assumption~\ref{ass:functions}. Let $\mathbf{1}\coloneqq(1,\ldots,1)^\top\in\mathbb{R}^d$, $I$ denote the identity matrix of dimension $d\times d$ and
		\begin{align*}
			u(t,x)&=\mathrm{e}^{-r(T-t)}\sum_{i=1}^d\sin(x_i)+d+1.
		\end{align*}
		For $r,\bar{\sigma},\kappa_y,\kappa_z\in\mathbb{R}_{>0}$, the coefficients of FBSDE~\eqref{eq:FBSDE} are given as
		\begin{align*}
			b(t,x,y,z)&=\kappa_y\bar{\sigma}T_y(y)\mathbf{1}+\kappa_zT_z(z)^\top,\\
			\sigma(t,x,y)&=\bar{\sigma}T_y(y)I,\\
			F(t,x,y,z)&=\begin{multlined}[t]
				r(y-(d+1))-\frac{1}{2}\mathrm{e}^{-r(T-t)}\bar{\sigma}^2u(t,x)^2\sum_{i=1}^d\sin(x_i)\\
				+\kappa_y\sum_{i=1}^dz_i+\kappa_z\mathrm{e}^{-2r(T-t)}\bar{\sigma}u(t,x)\sum_{i=1}^d\cos^2(x_i),
			\end{multlined}\\
			g(x)&=\sum_{i=1}^d\sin(x_i)+d+1,
		\end{align*}
		where
		\begin{align*}
			T_y(y)&\coloneqq\begin{cases}
				1,		& y<1,\\
				y,		& 1\leq y\leq2d+1,\\
				2d+1,	& y>2d+1,
			\end{cases}
			\quad\text{and}\quad
			T_z(z)\coloneqq\begin{cases}
				z,													& \lvert z\rvert\leq\bar{\sigma}(2d+1)\sqrt{d},\\
				\bar{\sigma}(2d+1)\sqrt{d}\frac{z}{\lvert z\rvert},	& \lvert z\rvert>\bar{\sigma}(2d+1)\sqrt{d}.
			\end{cases}
		\end{align*}
		Then, the solution of FBSDE~\eqref{eq:FBSDE} is
		\begin{align*}
			Y_t=u(t,X_t)
			\quad\text{and}\quad
			Z_t=v(t,X_t)
		\end{align*}
		with
		\begin{align*}
			 v(t,x)=\mathrm{e}^{-r(T-t)}\bar{\sigma}u(t,x)
			 \begin{pmatrix}
				\cos(x_1)\\
		 		\vdots\\
		 		\cos(x_d)
			 \end{pmatrix}.
		\end{align*}
		It is easy to see that Assumptions~\ref{ass:functions} and~\ref{ass:sol_PDE} are satisfied with $\alpha=1$.
		
		In the following, we first consider the case $d=r=\bar{\sigma}=\kappa_y=\kappa_z=T=1$ and $x_0=\frac{\pi}{4}$. Then, the conditions of \cite[Theorem~3]{Negyesi_Huang_Oosterlee_2026} are not fulfilled (see also the following Remark~\ref{rem:convergence}). Nevertheless, the conditions of Theorem~\ref{theorem:main} hold and thus we should see convergence. We train neural networks $\eta^*_N=(\hat y^*_{0,N},\hat\rho^*_N)$ for $N\in\{8,16,32,64,128,256\}$. In each training, we sample $\hat y^0_0$ from $\mathrm{Unif}(0,0.1)$ and each component of $\hat\gamma^0$ from $\mathcal{N}(0,1)$. We set the learning rate schedule by choosing $\lambda_0=0.01$, $k_\lambda=0.9549926$ and $n_\lambda=100$. The training process stops after $L=5000$ iterations. We chose $K_{\hat\rho^*_N}=6$. For each $N$, we repeat the training $50$ times. The results are given in the following figures.
		
		Figure~\ref{fig:rho_comparision} shows, for $t=0.5$, the fitted $\rho^*_N$'s (blue) in comparison with the true $v$ (black), grouped by the number of time steps $N$. Additionally, a histogram of $10^5$ realizations of $X_{0.5}$ is shown in each subplot. As can be seen, in the region where $X_{0.5}$ has high probability mass, the fitted $\hat\rho^*_N$'s become increasingly close to the true $v$ as the time grid is refined. The same holds for the fitted initial values $\hat y^*_{0,N}$: For each $N$, the squared differences to the true initial value $u(0,x_0)$, that is $\lvert\hat y^*_{0,N}-u(0,x_0)\rvert^2$, are displayed via boxplots in Figure~\ref{fig:y_init_comparision}. Lastly, Figure~\ref{fig:errors_1d} compares the $5\%$ quantile, median and $95\%$ quantile of terminal errors $\mathbb{E}\bigl[\lvert Y^{\hat\eta^*_N,\hat\pi_N}_T-g(X^{\hat\eta^*_N,\hat\pi_N}_T)\rvert^2\bigr]$  and the $L^2$-approximation error $\mathcal{E}_{\hat\eta^*_N,\hat\pi_N}$ in dependence of $N$. Additionally, the three components of $\mathcal{E}_{\hat\eta^*_N,\hat\pi_N}$ are given separately. All errors are calculated by averaging over the respective errors resulting from $25000$ simulated independent paths of the Brownian motion. Overall, we nearly obtain an empirical convergence rate of $\lvert\hat\pi_N\rvert^{1/2}$, that is the convergence rate from Theorem~\ref{theorem:main}.
		\begin{figure}
		    \centering
		    \includegraphics[width=\textwidth]{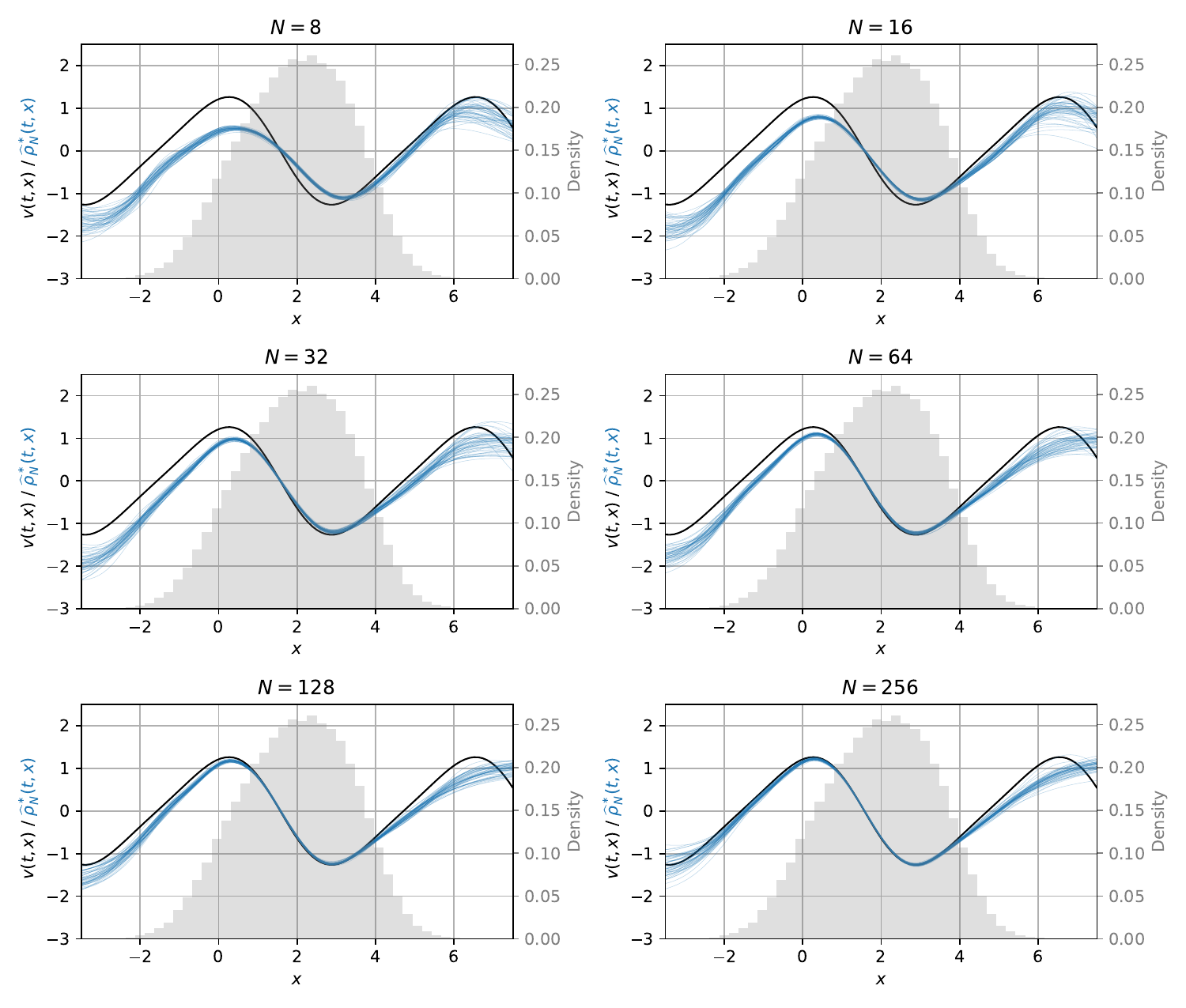}
		    \caption{Fitted $\hat\rho^*_N$'s vs.\@ true $v$ for $t=0.5$ grouped by the number of time steps~$N$.}
		    \label{fig:rho_comparision}
		\end{figure}
		\begin{figure}
		    \centering
		    \includegraphics[width=\textwidth]{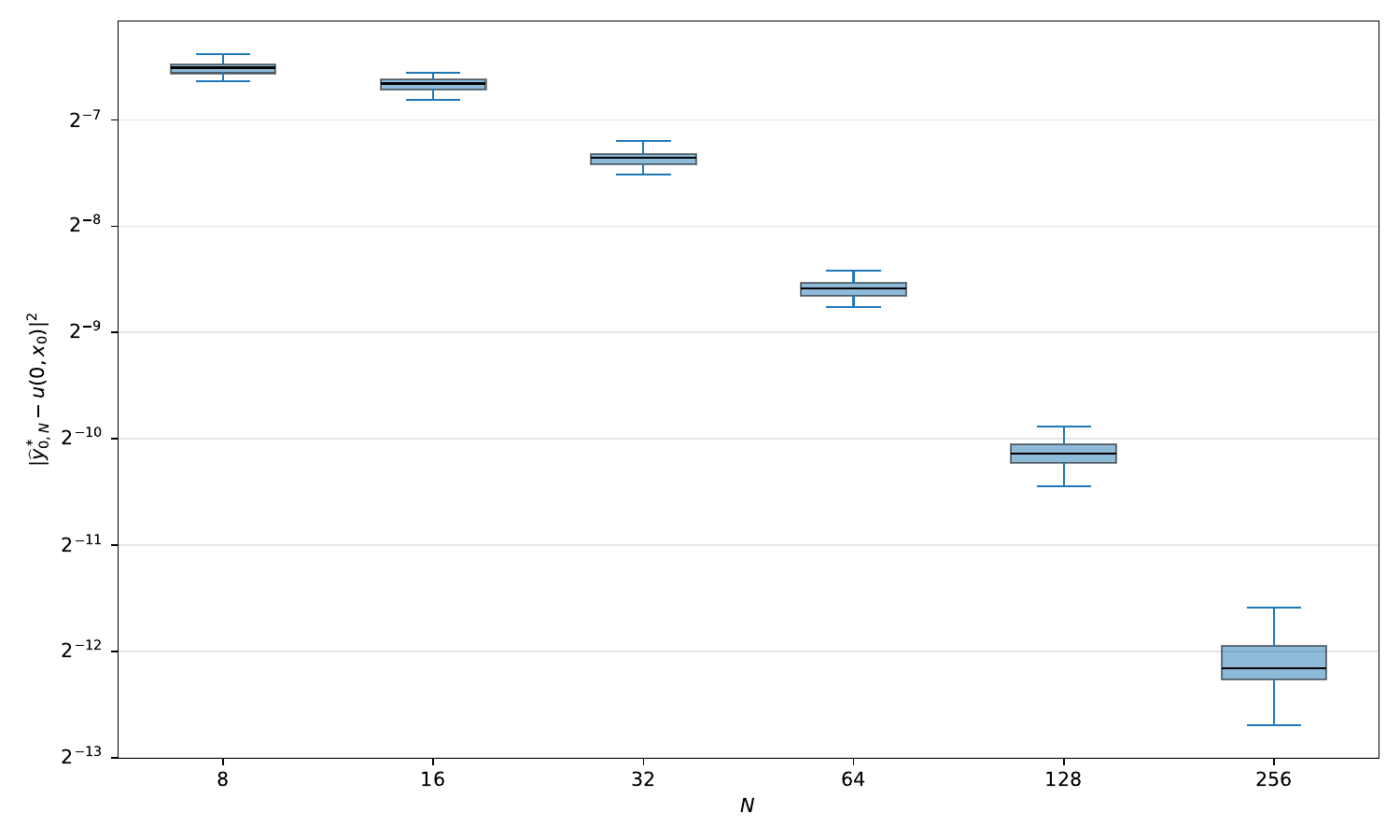}
		    \caption{Boxplots of the errors $\lvert\hat y^*_{0,N}-u(0,x_0)\rvert^2$ grouped by the number of time steps $N$.}
		    \label{fig:y_init_comparision}
		\end{figure}
		\begin{figure}
		    \centering
		    \includegraphics[width=\textwidth]{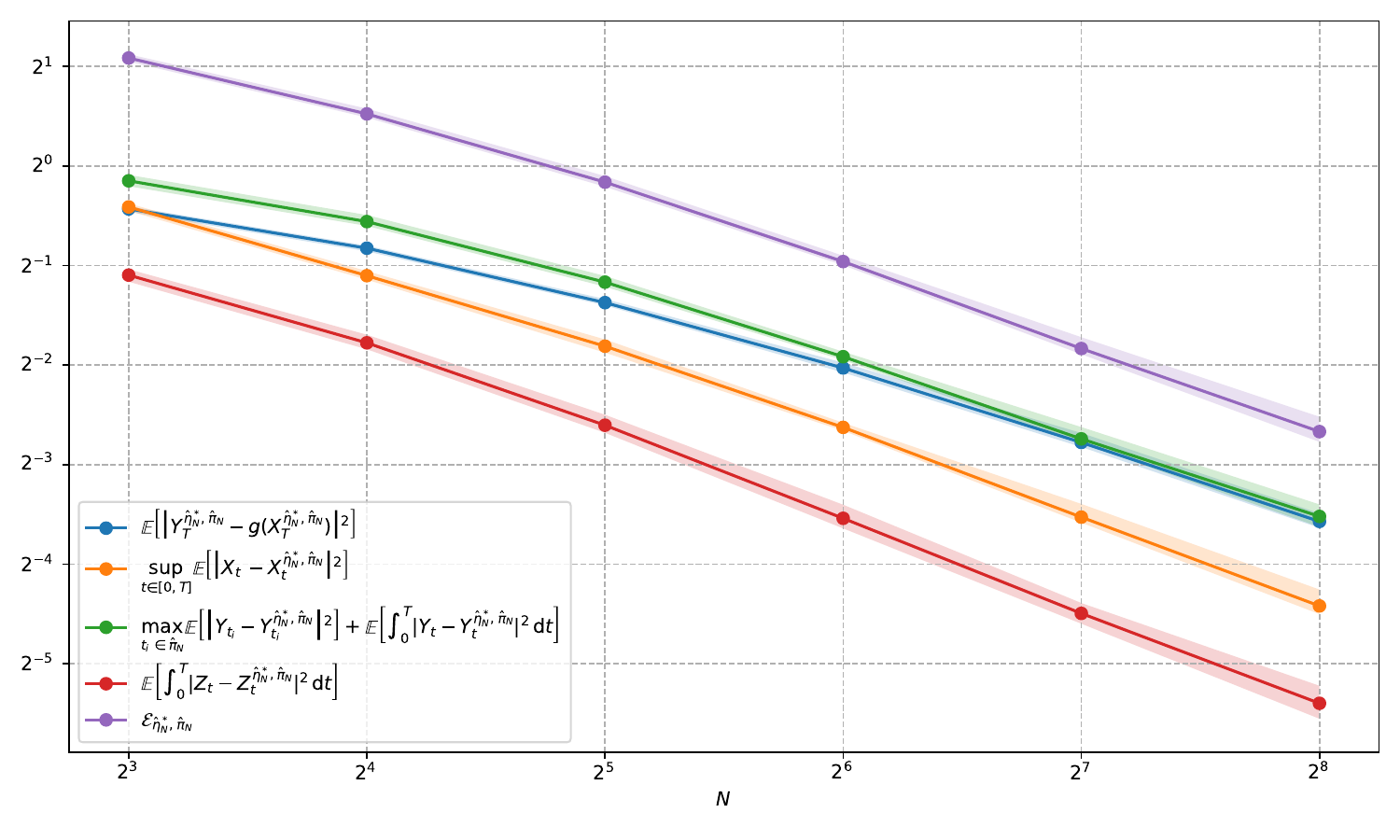}
		    \caption{Medians of the respective errors ($5\%$-$95\%$-quantile bands shaded) for $d=1$.}
		    \label{fig:errors_1d}
		\end{figure}
		
		Additionally, we consider the case $d=10$, $r=1$, $\bar{\sigma}=0.2$, $\kappa_y=0.2$, $\kappa_z=0.1$, $T=1$ and $x_0=(\frac{\pi}{4},\ldots,\frac{\pi}{4})^\top$. Again, the conditions of \cite[Theorem~3]{Negyesi_Huang_Oosterlee_2026} are not fulfilled, but the conditions of Theorem~\ref{theorem:main} are. We train networks $\eta^*_N=(\hat y^*_{0,N},\hat\rho^*_N)$ for $N\in\{8,16,32,64,128,256\}$. In each training, we sample $\hat y^0_0$ from $\mathrm{Unif}(0,0.1)$ and each component of $\hat\gamma^0$ from $\mathcal{N}(0,1)$. We set the learning rate schedule by choosing $\lambda_0=0.01$, $k_\lambda=0.99$ and $n_\lambda=100$. The training process stops after $L=2^{16}$ iterations. We choose $K_{\hat\rho^*_N}=100$. The terminal errors $\mathbb{E}\bigl[\lvert Y^{\hat\eta^*_N,\hat\pi_N}_T-g(X^{\hat\eta^*_N,\hat\pi_N}_T)\rvert^2\bigr]$ and the $L^2$-approximation error $\mathcal{E}_{\hat\eta^*_N,\hat\pi_N}$ including its three components are are given in Figure~\ref{fig:errors_10d}. Again, all errors are calculated by averaging over the respective errors resulting from $25000$ simulated independent paths of the Brownian motion. The empirical convergence rate is slightly worse than $\lvert\hat\pi_N\rvert^{1/2}$.
		\begin{figure}
		    \centering
		    \includegraphics[width=\textwidth]{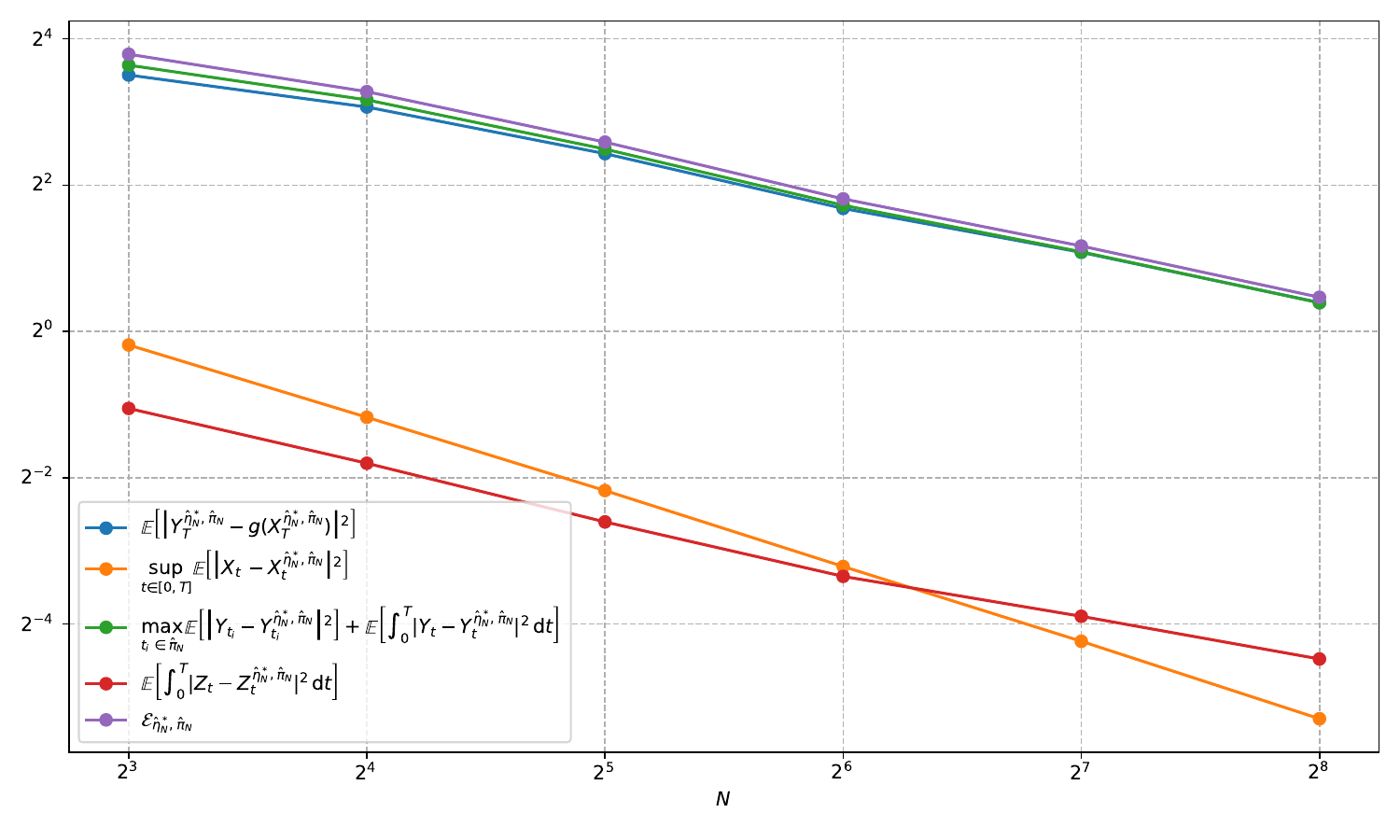}
		    \caption{Errors for $d=10$.}
		    \label{fig:errors_10d}
		\end{figure}
	\end{example}

	\begin{remark}\label{rem:convergence}
		For the given coefficients in the example above, it is a necessary condition for the applicability of \cite[Theorem~3]{Negyesi_Huang_Oosterlee_2026} that
		\begin{align*}
			\inf_{\lambda_1\in\mathbb{R}_{>0}}\kappa_z^2d^2\frac{\mathrm{e^{\lambda_1T}}}{\lambda_1}=\kappa_z^2d^2T\mathrm{e}<1.
		\end{align*}
		This follows since, in the setting of \cite[Theorem~3]{Negyesi_Huang_Oosterlee_2026}, a lower bound for $\bar{B}$ is easily derived as
		\begin{align*}
			\bar{B}>\kappa_z^2d^2\frac{\mathrm{e}^{\lambda_1T}}{\lambda_1},
		\end{align*}
		noting that the minimal squared Lipschitz constant of $b$ in $z$ equals $\kappa_z^2$, the minimal squared Lipschitz constant of $g$ in $x$ equals $d$, and $b$ and $\sigma$ are independent of $x$. Consequently, if
		\begin{align*}
			\kappa_z^2d^2T\mathrm{e}\geq1,
		\end{align*}	
		then \cite[Theorem~3]{Negyesi_Huang_Oosterlee_2026} cannot be applied. In particular, this is the case for both parameter sets considered in our example.
	\end{remark}

	\section{Proof of Theorem~\ref{theorem:main}}\label{sec:proof_main_theorem}
	\subsection{The Auxiliary Decoupled FBSDE}
	In this section we will prove Theorem~\ref{theorem:main}. To this end, we introduce an auxiliary decoupled FBSDE where the backward part and the control of the solution triple are connected with the forward part (and the time) solely via the same functions $u$, $u_x$ and $\sigma$ as the components of the solution as the original coupled FBSDE. This auxiliary decoupled FBSDE is then approximated by the same Euler-Maruyama scheme as the original FBSDE.
	
	In what follows, let $\tilde{x}\coloneqq(x^\top,\tilde{y})^\top$ and $\tilde{x}'\coloneqq(x'^\top,\tilde{y}')^\top$ with $x,x'\in\mathbb{R}^d$ and $\tilde{y},\tilde{y}'\in\mathbb{R}$, such that $\tilde{x},\tilde{x}'\in\mathbb{R}^{d+1}$. Fix $y_0\in\mathbb{R}$, $K_\rho,H_\rho\in\mathbb{R}_{>0}$ and $\rho\in\Theta_{K_\rho,H_\rho}$. We introduce the coefficients of the forward SDE of the auxiliary FBSDE as
	\begin{align*}
		b^\rho&\colon[0,T]\times\mathbb{R}^{d+1}\to\mathbb{R}^{d+1},\qquad
		(t,\tilde{x})\mapsto
		\begin{pmatrix}
			b(t,x,\tilde{y},\rho(t,x))\\
			F(t,x,\tilde{y},\rho(t,x))
		\end{pmatrix},
	\end{align*}
	and
	\begin{align*}
		\sigma^\rho&\colon[0,T]\times\mathbb{R}^{d+1}\to\mathbb{R}^{(d+1)\times d},\qquad
		(t,\tilde{x})\mapsto
		\begin{pmatrix}
			\sigma(t,x,\tilde{y})\\
			\rho(t,x)
		\end{pmatrix},
	\end{align*}
	Moreover, the generator $F^\rho\colon[0,T]\times\mathbb{R}^{d+1}\times\mathbb{R}\times(\mathbb{R}^d)^*\to\mathbb{R}$ of the backward SDE is given by
	\begin{align*}
		F^\rho(t,\tilde{x},y,z)&\coloneqq
		\begin{multlined}[t]
			F(t,x,y,\zeta(t,\tilde{x},y,z))
		+u_x(t,x)\bigl[b(t,x,\tilde{y},\rho(t,x))-b(t,x,y,\zeta(t,\tilde{x},y,z))\bigr]\\
		+\frac{1}{2}\mathrm{tr}\bigl[u_{xx}(t,x)\bigl((\sigma\sigma^\top)(t,x,\tilde{y})-(\sigma\sigma^\top)(t,x,y)\bigr)\bigr],
		\end{multlined}
	\end{align*}
	where the function $\zeta\colon[0,T]\times\mathbb{R}^{d+1}\times\mathbb{R}\times(\mathbb{R}^d)^*\to\mathbb{R}$ is given as
	\begin{align*}
		\zeta(t,\tilde{x},y,z)\coloneqq h(z)\sigma^{-1}(t,x,\tilde{y})\sigma(t,x,y)
	\end{align*}
	with truncation
	\begin{align*}
		h\colon(\mathbb{R}^d)^*\to(\mathbb{R}^d)^*,\quad
		z\mapsto
		\begin{cases}
			z & \text{if }\lvert z\rvert\leq K_\rho\vee K_\vartheta,\\
			(K_\rho\vee K_\vartheta)\frac{z}{\lvert z\rvert} & \text{if }\lvert z\rvert>K_\rho\vee K_\vartheta.
		\end{cases}
	\end{align*}
	Obviously, the truncation $h$ is bounded by $K_\rho\vee K_\vartheta$. Also, the truncation $h$ is Lipschitz continuous with constant $1$ as metric projection onto a closed convex set.

	Instead of the original FBSDE~\eqref{eq:FBSDE}, given the initial condition $\tilde{x}_0\coloneqq(x_0^\top,y_0)^\top$, we now look at the auxiliary FBSDE
	\begin{align}\label{eq:FBSDE-rho}
		\left\{ \begin{aligned} 
			\tilde{X}^\eta_t&=\tilde{x}_0+\int_0^tb^\rho(s,\tilde{X}^\eta_s)\,\mathrm{d}s+\int_0^t\sigma^\rho(s,\tilde{X}^\eta_s)\,\mathrm{d}W_s,\\
			Y^\eta_t&=g(X^\eta_T)-\int_t^TF^\rho(s,\tilde{X}^\eta_s,Y^\eta_s,Z^\eta_s)\mathrm{d}s-\int_t^TZ^\eta_s\,\mathrm{d}W_s,
		\end{aligned} \right.
	\end{align}
	where $Y^\eta_t$ and $Z^\eta_t$ do not couple into the equation for $\tilde{X}^\eta_t\coloneqq((X^\eta_t)^\top,\tilde{Y}^\eta_t)^\top$.
	
	We approximate the processes $X^\eta$, $Y^\eta$ and $Z^\eta$ by the triple $(X^{\eta,\pi},Y^{\eta,\pi},Z^{\eta,\pi})$, which we obtain from the Euler-Maruyama scheme~\eqref{eq:Euler}. Additionally, to approximate the process $\tilde{Y}^\eta$, we define the process $\tilde{Y}^{\eta,\pi}\equiv Y^{\eta,\pi}$. Setting $\tilde{X}^{\eta,\pi}\coloneqq((X^{\eta,\pi})^\top,(\tilde{Y}^{\eta,\pi})^\top)^\top$, this `extended' Euler-Maruyama scheme can easily be reformulated as
	\begin{align}\label{eq:Euler_extented}
		&\left\{\begin{aligned} 
			\tilde{X}^{\eta,\pi}_{t_0}&=\tilde{x}_0,\\
			\tilde{X}^{\eta,\pi}_{t_{i+1}}&=\tilde{X}^{\eta,\pi}_{t_i}+b^\rho(t_i,\tilde{X}^{\eta,\pi}_{t_i})\Delta_i+\sigma^\rho(t_i,\tilde{X}^{\eta,\pi}_{t_i})\Delta W_i,\;\forall i\in\{0,\ldots,N-1\},\\
			Z^{\eta,\pi}_{t_i}&=\rho(t_i,X^{\eta,\pi}_{t_i}),\;\forall i\in\{0,\ldots,N\},\\			
			\tilde{X}^{\eta,\pi}_t&=\tilde{X}^{\eta,\pi}_{\Pi(t)},\;Z^{\eta,\pi}_t=Z^{\eta,\pi}_{\Pi(t)},\;\forall t\in[0,T]\setminus\pi,\\
			Y^{\eta,\pi}_t&=\tilde{Y}^{\eta,\pi}_t,\;\forall t\in[0,T].
		\end{aligned}\right.
	\end{align}
	
	At this point it gets clear why the generator $F^\rho$ is defined in such a way: We can reasonably use the `same' approximation scheme to approximate the solutions of FBSDE~\eqref{eq:FBSDE} and FBSDE~\eqref{eq:FBSDE-rho}. In particular, at the grid points $t_i$, $i=0,\ldots,N$, it holds that
	\begin{align*}
		\zeta(t_i,\tilde{X}^{\eta,\pi}_{t_i},\tilde{Y}^{\eta,\pi}_{t_i},\rho(t_i,X^{\eta,\pi}_{t_i}))=\rho(t_i,X^{\eta,\pi}_{t_i}).
	\end{align*}
	Consequently, since $Y^{\eta,\pi}_{t_i}=\tilde{Y}^{\eta,\pi}_{t_i}$ and $Z^{\eta,\pi}_{t_i}=\rho(t_i,X^{\eta,\pi}_{t_i})$, we can simplify
	\begin{align}\label{eq:F_rho_for_approximation_triple}
		F^\rho(t_i,\tilde{X}^{\eta,\pi}_{t_i},Y^{\eta,\pi}_{t_i},Z^{\eta,\pi}_{t_i})
		&=F^\rho(t_i,\tilde{X}^{\eta,\pi}_{t_i},\tilde{Y}^{\eta,\pi}_{t_i},\rho(t_i,X^{\eta,\pi}_{t_i}))\nonumber\\
		&=F(t_i,X^{\eta,\pi}_{t_i},\tilde{Y}^{\eta,\pi}_{t_i},\rho(t_i,X^{\eta,\pi}_{t_i}))\nonumber\\
		&=F(t_i,X^{\eta,\pi}_{t_i},Y^{\eta,\pi}_{t_i},Z^{\eta,\pi}_{t_i}).
	\end{align}
	Hence, for $i\in\{0,\ldots,N-1\}$, we can rewrite
	\begin{align*}
		Y^{\eta,\pi}_{t_{i+1}}=Y^{\eta,\pi}_{t_i}+F^\rho(t_i,\tilde{X}^{\eta,\pi}_{t_i},Y^{\eta,\pi}_{t_i},Z^{\eta,\pi}_{t_i})\Delta_i+Z^{\eta,\pi}_{t_i}\Delta W_i,
	\end{align*}
	i.\,e., the tuple $(\tilde{X}^{\eta,\pi},Y^{\eta,\pi},Z^{\eta,\pi})$ can be considered as the deep BSDE approximation for the decoupled FBSDE with coefficients $b^\rho$, $\sigma^\rho$ and $F^\rho$. We now first study this decoupled FBSDE and prove the following theorem.
	
	\begin{theorem}\label{theorem:auxFBSDE-sol}
		FBSDE~\eqref{eq:FBSDE-rho} admits a unique adapted solution $(\tilde{X}^\eta,Y^\eta,Z^\eta)$. Moreover, for each $t\in[0,T]$, it holds that
		\begin{align*}
			Y^\eta_t=u(t,X^\eta_t)
			\quad\text{and}\quad
			Z^\eta_t=\vartheta(t,X^\eta_t,\tilde{Y}^\eta_t).
		\end{align*}
	\end{theorem}
	
	To ensure that the FBSDE~\eqref{eq:FBSDE-rho} is uniquely solvable certain regularity conditions concerning $b^\rho$, $\sigma^\rho$ and $F^\rho$ are crucial, which we will establish next. For later use, it is important to keep track of the dependence of the constants on $K_\rho$ and $H_\rho$.
	\begin{remark}
		In what follows, we will repeatedly use the following basic inequalities without explicitly referring to them:
		\begin{itemize}
			\item[(i)]
				Let $m\in\mathbb{N}$ and $a_1,\ldots,a_m\in\mathbb{R}$. It holds that
				\begin{align*}
					\left(\sum_{i=1}^ma_i\right)^2\leq m\sum_{i=1}^ma_i^2.
				\end{align*}				
			\item[(ii)]
				For any two $\tilde{x}=(x^\top,\tilde{y})^\top,\tilde{x}'=(x'^\top,\tilde{y}')^\top\in\mathbb{R}^{d+1}$, we have that
				\begin{align*}
					\lvert x-x'\rvert+\lvert\tilde{y}-\tilde{y}'\rvert
					\leq\sqrt{2}\lvert\tilde{x}-\tilde{x}'\rvert.
				\end{align*}
		\end{itemize}
	\end{remark}	

	\begin{lem}\label{lem:Lipschitz_driver_diffusion}
		Let
		\begin{align*}
			H_{b^\rho}\coloneqq(\sqrt{2}+H_\rho)\sqrt{H_b^2+H_F^2}
			\quad\text{and}\quad
			H_{\sigma^\rho}\coloneqq\sqrt{2H_\sigma^2+H_\rho^2}.
		\end{align*}
		For any $(t,\tilde{x}),(t',\tilde{x}')\in[0,T]\times\mathbb{R}^{d+1}$ it holds that
		\begin{align*}
			\lvert b^\rho(t,\tilde{x})-b^\rho(t',\tilde{x}')\rvert&\leq H_{b^\rho}\bigl(\lvert t-t'\rvert^{1/2}+\lvert \tilde{x}-\tilde{x}'\rvert\bigr),\\
			\lvert\sigma^\rho(t,\tilde{x})-\sigma^\rho(t',\tilde{x}')\rvert&\leq H_{\sigma^\rho}\bigl(\lvert t-t'\rvert^{1/2}+\lvert \tilde{x}-\tilde{x}'\rvert\bigr),
		\end{align*}
		i.\,e., the functions $b^\rho$ and $\sigma^\rho$ are $\frac{1}{2}$-Hölder continuous in $t$ and Lipschitz continuous in $\tilde{x}$ with constants $H_{b^\rho}$ and $H_{\sigma^\rho}$.
	\end{lem}
	\begin{proof}
		The definition of the Euclidean norm in connection with the Hölder and Lipschitz conditions on $b$, $F$, and~$\rho$ imply
		\begin{align*}
			\lvert b^\rho(t,\tilde{x})-b^\rho(t',\tilde{x}')\rvert^2
			&=\begin{multlined}[t]
				\lvert b(t,x,\tilde{y},\rho(t,x))-b(t',x',\tilde{y}',\rho(t',x'))\rvert^2\\
				+\lvert F(t,x,\tilde{y},\rho(t,x))-F(t',x',\tilde{y}',\rho(t',x'))\rvert^2
			\end{multlined}\\
			&\leq(H_b^2+H_F^2)\bigl(\lvert t-t'\rvert^{1/2}+\lvert x-x'\rvert+\lvert\tilde{y}-\tilde{y}'\rvert+\lvert\rho(t,x)-\rho(t',x')\rvert\bigr)^2\\
			&\leq(H_b^2+H_F^2)(\sqrt{2}+H_\rho)^2\bigl(\lvert t-t'\rvert^{1/2}+\lvert\tilde{x}-\tilde{x}'\rvert\bigr)^2.
		\end{align*}
		Taking square roots on both sides shows the first inequality. Analogously, now taking the definition of the Frobenius norm into account and applying the $\frac{1}{2}$-Hölder and Lipschitz continutity of $\sigma$ and $\rho$, we obtain
		\begin{align*}
			\lvert\sigma^\rho(t,\tilde{x})-\sigma^\rho(t,\tilde{x}')\rvert^2
			&=\lvert\sigma(t,x,\tilde{y})-\sigma(t',x',\tilde{y}')\rvert^2+\lvert\rho(t,x)-\rho(t',x')\rvert^2\\
			&\leq H_\sigma^2\bigl(\lvert t-t'\rvert^{1/2}+\lvert x-x'\rvert+\lvert \tilde{y}-\tilde{y}'\rvert\bigr)^2+H_\rho^2\bigl(\lvert t-t'\rvert^{1/2}+\lvert x-x'\rvert\bigr)^2\\
			&\leq(2H_\sigma^2+H_\rho^2)\bigl(\lvert t-t'\rvert^{1/2}+\lvert \tilde{x}-\tilde{x}'\rvert\bigr)^2.
		\end{align*}
		Taking square roots on both sides shows the second inequality.
	\end{proof}
	\begin{lem}\label{lem:linear_growth}
		Define
		\begin{align*}
		 	L_{b^\rho}&\coloneqq\sqrt{12H_F^2\vee\bigl(6H_F^2(T+K_\rho^2)+K_b^2+2F(0,0,0,0)^2\bigr)},\\
		 	K_{\sigma^\rho}&\coloneqq\sqrt{K_\sigma^2+K_\rho^2}.
		\end{align*}
		For any $(t,\tilde{x})\in[0,T]\times\mathbb{R}^{d+1}$ it holds that
		\begin{align*}
			\lvert b^\rho(t,\tilde{x})\rvert^2\leq L_{b^\rho}^2(1+\lvert\tilde{x}\rvert^2)
			\quad\text{and}\quad
			\lvert\sigma^\rho(t,\tilde{x})\rvert^2\leq K_{\sigma^\rho}^2.
		\end{align*}
	\end{lem}
	\begin{proof}
		Using the Hölder and Lipschitz conditions on $F$ as well as the boundedness of $b$ and $\rho$, we obtain
		\begin{align*}
			\lvert b^\rho(t,\tilde{x})\rvert^2
			&=\lvert b(t,x,\tilde{y},\rho(t,x))\rvert^2+\lvert F(t,x,\tilde{y},\rho(t,x))\rvert^2\\
			&\leq K_b^2+2\lvert F(t,x,\tilde{y},\rho(t,x))-F(0,0,0,0)\rvert^2+2F(0,0,0,0)^2\\
			&\leq K_b^2+2H_F^2\bigl(t^{1/2}+\lvert x\rvert+\lvert\tilde{y}\rvert+\lvert\rho(t,x)\rvert\bigr)^2+2F(0,0,0,0)^2\\
			&\leq12H_F^2\lvert\tilde{x}\rvert^2+6H_F^2(T+K_\rho^2)+K_b^2+2F(0,0,0,0)^2.
		\end{align*}
		This shows the linear growth condition of $b^\rho$. The boundedness of $\sigma^\rho$ follows directly from the boundedness of $\sigma$ and $\rho$, i.\,e.,
		\begin{align*}
			\lvert\sigma^\rho(t,\tilde{x})\rvert^2
			&=\lvert\sigma(t,x,\tilde{y})\rvert^2+\lvert\rho(t,x)\rvert^2
			\leq K_\sigma^2+K_\rho^2.\qedhere
		\end{align*}
	\end{proof}
	
	\begin{lem}\label{lem:Lipschitz}
		Let $\rho\in\Theta_{K_\rho,H_\rho}$. Define
		\begin{align*}
			H_1&\coloneqq(H_F+K_{u_x}H_b)\bigl(1+H_\zeta^{(y,z)}\bigr)+\frac{1}{2}K_{u_{xx}}H_{\sigma\sigma^\top},\\
			H_2&\coloneqq(H_F+K_{u_x}H_b)\bigl(1+H_\zeta^{(t,x)}\bigr)+K_{u_x}H_b(\sqrt{2}+H_\rho)+2K_bH_{u_x}+\frac{1+\sqrt{2}}{2}K_{u_{xx}}H_{\sigma\sigma^\top},\\
			H_3&\coloneqq K_{\sigma\sigma^\top}H_{u_{xx}},
		\end{align*}
		where
		\begin{align*}
			H_\zeta^{(y,z)}\coloneqq K_{\sigma^{-1}}(K_\sigma\vee(K_\rho\vee K_\vartheta)H_\sigma)
			\quad\text{and}\quad
			H_\zeta^{(t,x)}\coloneqq(K_\rho\vee K_\vartheta)(\sqrt{2}K_\sigma H_{\sigma^{-1}}+K_{\sigma^{-1}}H_\sigma).
		\end{align*}
		\begin{itemize}
			\item[(a)]
				For all $(t,\tilde{x})\in[0,T]\times\mathbb{R}^{d+1}$ and any two $(y,z),(y',z')\in\mathbb{R}\times(\mathbb{R}^d)^*$ it holds that
				\begin{align*}
					\lvert F^\rho(t,\tilde{x},y,z)-F^\rho(t,\tilde{x},y',z')\rvert\leq H_1\bigl(\lvert y-y'\rvert+\lvert z-z'\rvert\bigr),
				\end{align*}
				i.\,e., $F^\rho$ is Lipschitz in $(y,z)$ uniformly in $(t,\tilde{x})$ with constant $H_1$.
			\item[(b)]
				For all $(y,z)\in\mathbb{R}\times(\mathbb{R}^d)^*$ and any two $(t,\tilde{x}),(t',\tilde{x}')\in[0,T]\times\mathbb{R}^{d+1}$ it holds that
				\begin{align*}
					\lvert F^\rho(t,\tilde{x},y,z)-F^\rho(t',\tilde{x}',y,z)\rvert\leq H_2\bigl(\lvert t-t'\rvert^{1/2}+\lvert \tilde{x}-\tilde{x}'\rvert\bigr)+H_3\bigl(\lvert t-t'\rvert^{\alpha/2}+\lvert \tilde{x}-\tilde{x}'\rvert^\alpha\bigr).
				\end{align*}
		\end{itemize}
	\end{lem}
	
	\begin{proof}
		\begin{itemize}
			\item[(a)]
				Using the shorthand notation
				\begin{align*}
					\tilde\sigma^{-1}&\coloneqq\sigma^{-1}(t,x,\tilde{y}),&
					\sigma&\coloneqq\sigma(t,x,y),&
					\sigma'&\coloneqq\sigma(t,x,y'),
				\end{align*}	
				we obtain
				\begin{align*}
					\lvert\zeta(t,\tilde{x},y,z)-\zeta(t,\tilde{x},y',z')\rvert
					&=\lvert h(z)\tilde\sigma^{-1}\sigma-h(z')\tilde\sigma^{-1}\sigma'\rvert\\
					&=\lvert(h(z)-h(z'))\tilde\sigma^{-1}\sigma+h(z')\tilde\sigma^{-1}(\sigma-\sigma')\rvert\\
					&\leq K_\sigma K_{\sigma^{-1}}\lvert z-z'\rvert+(K_\rho\vee K_\vartheta)K_{\sigma^{-1}}H_\sigma\lvert y-y'\rvert\\
					&\leq H_\zeta^{(y,z)}\bigl(\lvert y-y'\rvert+\lvert z-z'\rvert\bigr).
				\end{align*}
				With this bound at hand, it easily follows that
				\begin{align*}
					\lvert F(t,x,y,\zeta(t,\tilde{x},y,z))-F(t,x,y',\zeta(t,\tilde{x},y',z'))\rvert
					&\leq H_F\bigl(1+H_\zeta^{(y,z)}\bigr)\bigl(\lvert y-y'\rvert+\lvert z-z'\rvert\bigr),\\
					\lvert b(t,x,y,\zeta(t,\tilde{x},y,z))-b(t,x,y',\zeta(t,\tilde{x},y',z'))\rvert
					&\leq H_b\bigl(1+H_\zeta^{(y,z)}\bigr)\bigl(\lvert y-y'\rvert+\lvert z-z'\rvert\bigr).
				\end{align*}
				Thus, using the shorthand notation
				\begin{align*}
					b&\coloneqq b(t,x,y,\zeta(t,\tilde{x},y,z)),&
					b'&\coloneqq b(t,x,y',\zeta(t,\tilde{x},y',z')),&
					b_\rho&\coloneqq b(t,x,\tilde{y},\rho(t,x)),
				\end{align*}	
				we have
				\begin{align*}
					\lvert u_x(t,x)(b_\rho-b)-u_x(t,x)(b_\rho-b')\rvert
					&=\lvert u_x(t,x)\rvert\lvert b'-b\rvert\\
					&\leq K_{u_x}H_b\bigl(1+H_\zeta^{(y,z)}\bigr)\bigl(\lvert y-y'\rvert+\lvert z-z'\rvert\bigr).
				\end{align*}
				Lastly, it holds that
				\begin{align*}
					&\begin{multlined}[t]
						\biggl\lvert\frac{1}{2}\mathrm{tr}\bigl[u_{xx}(t,x)\bigl((\sigma\sigma^\top)(t,x,\tilde{y})-(\sigma\sigma^\top)(t,x,y)\bigr)\bigr]\\
						-\frac{1}{2}\mathrm{tr}\bigl[u_{xx}(t,x)\bigl((\sigma\sigma^\top)(t,x,\tilde{y})-(\sigma\sigma^\top)(t,x,y')\bigr)\bigr]\biggr\rvert
					\end{multlined}\\
					&=\frac{1}{2}\bigl\lvert\mathrm{tr}\bigl[u_{xx}(t,x)\bigl((\sigma\sigma^\top)(t,x,y')-(\sigma\sigma^\top)(t,x,y)\bigr)\bigr]\bigr\rvert\\
					&\leq\frac{1}{2}\lvert u_{xx}(t,x)\rvert\bigl\lvert(\sigma\sigma^\top)(t,x,y')-(\sigma\sigma^\top)(t,x,y)\bigr\rvert\\
					&\leq\frac{1}{2}K_{u_{xx}}H_{\sigma\sigma^\top}\lvert y-y'\rvert.
				\end{align*}
				The statement now follows by applying the triangle inequality to the difference $\lvert F^\rho(t,\tilde{x},y,z)-F^\rho(t,\tilde{x},y',z')\rvert$ and summing up the above bounds.
			\item[(b)]
				Using the shorthand notation
				\begin{align*}
					\sigma&\coloneqq\sigma(t,x,y),&\sigma'&\coloneqq\sigma(t',x',y),&\tilde\sigma^{-1}&\coloneqq\sigma^{-1}(t,x,\tilde{y}),&(\tilde\sigma^{-1})'&\coloneqq\sigma^{-1}(t',x',\tilde{y}'),
				\end{align*}	
				we obtain
				\begin{align*}
					\lvert\zeta(t,\tilde{x},y,z)-\zeta(t',\tilde{x}',y,z)\rvert
					&=\lvert h(z)\tilde\sigma^{-1}\sigma-h(z)(\tilde\sigma^{-1})'\sigma'\rvert\\
					&=\lvert h(z)\rvert\bigl\lvert\bigl(\tilde\sigma^{-1}-(\tilde\sigma^{-1})'\bigr)\sigma+(\tilde\sigma^{-1})'(\sigma-\sigma')\bigr\rvert\\
					&\leq(K_\rho\vee K_\vartheta)(\sqrt{2}K_\sigma H_{\sigma^{-1}}+K_{\sigma^{-1}}H_\sigma)\bigl(\lvert t-t'\rvert^{1/2}+\lvert\tilde{x}-\tilde{x}'\rvert\bigr)\\
					&=H_\zeta^{(t,x)}\bigl(\lvert t-t'\rvert^{1/2}+\lvert\tilde{x}-\tilde{x}'\rvert\bigr).
				\end{align*}
				Then, we have
				\begin{align*}
					\lvert F(t,x,y,\zeta(t,\tilde{x},y,z))-F(t',x',y,\zeta(t',\tilde{x}',y,z))\rvert
					&\leq H_F\bigl(1+H_\zeta^{(t,x)}\bigr)\bigl(\lvert t-t'\rvert^{1/2}+\lvert\tilde{x}-\tilde{x}'\rvert\bigr),\\
					\lvert b(t,x,y,\zeta(t,\tilde{x},y,z))-b(t',x',y,\zeta(t',\tilde{x}',y,z))\rvert
					&\leq H_b\bigl(1+H_\zeta^{(t,x)}\bigr)\bigl(\lvert t-t'\rvert^{1/2}+\lvert\tilde{x}-\tilde{x}'\rvert\bigr).
				\end{align*}
				Furthermore, it holds that
				\begin{align*}
					&\lvert b(t,x,\tilde{y},\rho(t,x))-b(t',x',\tilde{y}',\rho(t',x'))\rvert\\
					&\leq H_b\bigl(\lvert t-t'\rvert^{1/2}+\lvert x-x'\rvert+\lvert\tilde{y}-\tilde{y}'\rvert+\lvert\rho(t,x)-\rho(t',x')\rvert\bigr)\\
					&\leq H_b(\sqrt{2}+H_\rho)\bigl(\lvert t-t'\rvert^{1/2}+\lvert\tilde{x}-\tilde{x}'\rvert\bigr).
				\end{align*}
				Now, using the shorthand notation
				\begin{align*}
					b_\rho&\coloneqq b(t,x,\tilde{y},\rho(t,x)),&b&\coloneqq b(t,x,y,\zeta(t,\tilde{x},y,z)),\\
					b_\rho'&\coloneqq b(t',x',\tilde{y}',\rho(t',x')),&b'&\coloneqq b(t',x',y,\zeta(t',\tilde{x}',y,z)),
				\end{align*}
				we obtain
				\begin{align*}
					&\lvert u_x(t,x)(b_\rho-b)-u_x(t',x')(b_\rho'-b')\rvert\\
					&=\lvert u_x(t,x)(b_\rho-b_\rho'+b'-b)+(u_x(t,x)-u_x(t',x'))(b_\rho'-b')\rvert\\
					&\leq\begin{multlined}[t]
						\bigl[K_{u_x}H_b\bigl(1+\sqrt{2}+H_\rho+H_\zeta^{(t,x)}\bigr)+2K_bH_{u_x}\bigr]\bigl(\lvert t-t'\rvert^{1/2}+\lvert\tilde{x}-\tilde{x}'\rvert\bigr).
					\end{multlined}
				\end{align*}
				Similarly, now abbreviating
				\begin{align*}
					\sigma\sigma^\top&\coloneqq(\sigma\sigma^\top)(t,x,y),&\widetilde{\sigma\sigma}^\top&\coloneqq(\sigma\sigma^\top)(t,x,\tilde y),\\
					(\sigma\sigma^\top)'&\coloneqq(\sigma\sigma^\top)(t',x',y),&(\widetilde{\sigma\sigma}^\top)'&\coloneqq(\sigma\sigma^\top)(t',x',\tilde{y}'),
				\end{align*}
				we obtain
				\begin{align*}
					&\left\lvert\frac{1}{2}\mathrm{tr}\bigl[u_{xx}(t,x)\bigl(\widetilde{\sigma\sigma}^\top-\sigma\sigma^\top\bigr)\bigr]-\frac{1}{2}\mathrm{tr}\bigl[u_{xx}(t',x')\bigl((\widetilde{\sigma\sigma}^\top)'-(\sigma\sigma^\top)'\bigr)\bigr]\right\rvert\\
					&\leq\begin{multlined}[t]
						\frac{1}{2}\bigl\lvert\mathrm{tr}\bigl[\bigl(u_{xx}(t,x)-u_{xx}(t',x')\bigr)\bigl(\widetilde{\sigma\sigma}^\top-\sigma\sigma^\top\bigr)\bigr]\bigr\rvert\\
						+\frac{1}{2}\bigl\lvert\mathrm{tr}\bigl[u_{xx}(t',x')\bigl(\widetilde{\sigma\sigma}^\top-(\widetilde{\sigma\sigma}^\top)'+(\sigma\sigma^\top)'-\sigma\sigma^\top\bigr)\bigr]\bigr\rvert
					\end{multlined}\\
					&\leq\frac{1}{2}\lvert u_{xx}(t,x)-u_{xx}(t',x')\rvert\lvert\widetilde{\sigma\sigma}^\top-\sigma\sigma^\top\rvert+\frac{1}{2}\lvert u_{xx}(t',x')\rvert\lvert\widetilde{\sigma\sigma}^\top-(\widetilde{\sigma\sigma}^\top)'+(\sigma\sigma^\top)'-\sigma\sigma^\top\rvert\\
					&\leq\begin{multlined}[t]
						K_{\sigma\sigma^\top}H_{u_{xx}}\bigl(\lvert t-t'\rvert^{\alpha/2}+\lvert\tilde{x}-\tilde{x}'\rvert^\alpha\bigr)+\frac{1}{2}K_{u_{xx}}H_{\sigma\sigma^\top}(1+\sqrt{2})\bigl(\lvert t-t'\rvert^{1/2}+\lvert\tilde{x}-\tilde{x}'\rvert\bigr).
					\end{multlined}
				\end{align*}				
				The statement now follows by applying the triangle inequality to the difference $\lvert F^\rho(t,\tilde{x},y,z)-F^\rho(t',\tilde{x}',y,z)\rvert$ and summing up the above bounds.\qedhere
		\end{itemize}
	\end{proof}
	
	\begin{proof}[Proof of Theorem~\ref{theorem:auxFBSDE-sol}]
		Since $b^\rho$ and $\sigma^\rho$ are uniformly Lipschitz in $\tilde{x}$ (Lemma~\ref{lem:Lipschitz_driver_diffusion}) and fulfill linear growth conditions in $\tilde{x}$ (Lemma~\ref{lem:linear_growth}, for $\sigma^\rho$ the boundedness implies the linear growth condition), the forward part of FBSDE~\eqref{eq:FBSDE-rho} has a unique solution $\tilde{X}^\eta$ \cite[Theorem~3.1 of Chapter~2]{Mao_2011}. In particular, $\tilde{X}^\eta$ is continuous and $(\mathcal{F}_t)_{t\geq0}$-adapted and consequently also $(\mathcal{F}_t)_{t\geq0}$-progressively measurable. Moreover, by \cite[Lemma~3.2 of Chapter~2]{Mao_2011} in connection with Lemma~\ref{lem:linear_growth}, it holds that $\mathbb{E}\bigl[\sup_{t\in[0,T]}\lvert\tilde{X}^\eta_t\rvert^2\bigr]<\infty$. Thus, also the backward part admits a unique solution \cite[Theorem~4.2 of Chapter~1]{Ma_Yong_2007} since $F^\rho$, considered as a function
		\begin{align*}
			F^\rho\colon[0,T]\times\Omega\times\mathbb{R}\times(\mathbb{R}^d)^*\to\mathbb{R},\quad
			F^\rho(t,\omega,y,z)=F^\rho(t,\tilde{X}^\eta_t(\omega),y,z),
		\end{align*}
		is Lipschitz in $(y,z)$ uniformly in $(t,\omega)$ by Lemma~\ref{lem:Lipschitz}(a) and satisfies
		\begin{align}\label{eq:int_cond_X_eta_tilde}
			&\mathbb{E}\left[\lvert g(\tilde{X}^\eta_T)\rvert^2+\int_0^T\lvert F^\rho(t,\tilde{X}^\eta_t,0,0)\rvert^2\,\mathrm{d}t\right]\nonumber\\
			&\leq\begin{multlined}[t]
				2H_g^2\mathbb{E}\left[\lvert\tilde{X}^\eta_T\rvert^2\right]+2H_g^2g(0)^2+2TF^\rho(0,0,0,0)^2\\
				+8\mathbb{E}\left[\int_0^T H_2^2t+H_2^2\lvert\tilde{X}^\eta_t\rvert^2+H_3^2t^\alpha+H_3^2\lvert\tilde{X}^\eta_t\rvert^{2\alpha}\,\mathrm{d}t\right]
			\end{multlined}\nonumber\\
			&<\infty
		\end{align}
		by the Lipschitz continuity of $g$, Lemma~\ref{lem:Lipschitz}(b) and the integrability properties of $\tilde{X}^\eta$.
		
		Now, fix some $t\in[0,T]$ and apply Itô's formula to $u(t,X^\eta_t)$. Additionally, using the identities for $u_t(x,X_s)$ and $u(T,X_T)$ given by the parabolic PDE~\eqref{eq:PDE}	and noting that
		\begin{align*}
			v(s,X^\eta_s)
			&=u_x(s,X^\eta_s)\sigma(s,X^\eta_s,u(s,X^\eta_s))\\
			&=h(\vartheta(s,X^\eta_s,\tilde{Y}^\eta_s))\sigma^{-1}(s,X^\eta_s,\tilde{Y}^\eta_s)\sigma(s,X^\eta_s,u(s,X^\eta_s))\\
			&=h(Z^\eta_s)\sigma^{-1}(s,X^\eta_s,\tilde{Y}^\eta_s)\sigma(s,X^\eta_s,Y^\eta_s)\\
			&=\zeta(s,\tilde{X}^\eta_s,Y^\eta_s,Z^\eta_s),
		\end{align*}
		since $\vartheta(s,X^\eta_s,\tilde{Y}^\eta_s)$ is bounded by $K_\vartheta\leq K_\rho\vee K_\vartheta$, it follows that
		\begin{align*}
			Y^\eta_t
			&=\begin{multlined}[t][0.7\displaywidth]
				u(T,X^\eta_T)-\int_t^Tu_s(s,X^\eta_s)+u_{x}(s,X^\eta_s)b(s,X^\eta_s,\tilde{Y}^\eta_s,\rho(s,X^\eta_s))\\
				+\frac{1}{2}\mathrm{tr}\bigl[u_{xx}(s,X^\eta_s)(\sigma\sigma^\top)(s,X^\eta_s,\tilde{Y}^\eta_s)\bigr]\,\mathrm{d}s-\int_t^Tu_{x}(s,X^\eta_s)\sigma(s,X^\eta_s,\tilde{Y}^\eta_s)\,\mathrm{d}W_s
			\end{multlined}\\
			&=\begin{multlined}[t]
				g(X^\eta_T)-\int_t^TF(s,X^\eta_s,u(s,X^\eta_s),v(s,X^\eta_s))\\
				+u_x(s,X^\eta_s)\bigl[b(s,X^\eta_s,\tilde{Y}^\eta_s,\rho(s,X^\eta_s))-b(s,X^\eta_s,u(s,X^\eta_s),v(s,X^\eta_s))\bigr]\\
				+\frac{1}{2}\mathrm{tr}\bigl[u_{xx}(s,X^\eta_s)\bigl((\sigma\sigma^\top)(s,X^\eta_s,\tilde{Y}^\eta_s)-(\sigma\sigma^\top)(s,X^\eta_s,u(s,X^\eta_s)\bigr)\bigr]\,\mathrm{d}s-\int_t^TZ^\eta_s\,\mathrm{d}W_s
			\end{multlined}\\
			&=\begin{multlined}[t]
				g(X^\eta_T)-\int_t^TF(s,X^\eta_s,Y^\eta_s,\zeta(s,\tilde{X}^\eta_s,Y^\eta_s,Z^\eta_s))\\
				+u(s,X^\eta_s)\bigl[b(s,X^\eta_s,\tilde{Y}^\eta_s,\rho(s,X^\eta_s))-b(s,X^\eta_s,Y^\eta_s,\zeta(s,\tilde{X}^\eta_s,Y^\eta_s,Z^\eta_s))\bigr]\\
				+\frac{1}{2}\mathrm{tr}\bigl[u_{xx}(s,X^\eta_s)\bigl((\sigma\sigma^\top)(s,X^\eta_s,\tilde{Y}^\eta_s)-(\sigma\sigma^\top)(s,X^\eta_s,Y^\eta_s\bigr)\bigr]\,\mathrm{d}s-\int_t^TZ^\eta_s\,\mathrm{d}W_s
			\end{multlined}\\
			&=g(X^\eta_T)-\int_t^TF^\rho(s,\tilde{X}^\eta_s,Y^\eta_s,Z^\eta_s)\,\mathrm{d}s-\int_t^TZ^\eta_s\,\mathrm{d}W_s.
		\end{align*}
		That is $Y^\eta$ and $Z^\eta$ fulfill the backward SDE of FBSDE~\eqref{eq:FBSDE-rho}. Moreover, since $u$ and $\vartheta$ fulfill Hölder and Lipschitz conditions and are consequently continuous, the processes $Y^\eta$ and $Z^\eta$ are $(\mathcal{F}_t)_{t\geq0}$-progressively measurable as functions of $\tilde{X}^\eta$. Also $Y^\eta$ is continuous. Lastly, the necessary integrability conditions are fulfilled. Indeed, for $Y^\eta$, we obtain from the integrability properties of $\tilde{X}^\eta$ that
		\begin{align*}
			\mathbb{E}\biggl[\sup_{t\in[0,T]}\lvert Y^\eta_t\rvert^2\biggr]
			\leq4H_u^2T+4H_u^2\mathbb{E}\biggl[\sup_{t\in[0,T]}\lvert X^\eta_t\rvert^2\biggr]+2u(0,0)^2
			<\infty.
		\end{align*}	For $Z^\eta$, the required integrability condition follows directly from the boundedness of $\vartheta$.
	\end{proof}
	
	\begin{lem}\label{lem:supremum-Xtilde}
		It holds that
		\begin{align*}
			\mathbb{E}\biggl[\sup_{t\in[0,T]}\lvert\tilde{X}^\eta_t\rvert^2\biggr]\leq(1+3\lvert\tilde{x}_0\rvert^2)\mathrm{e}^{3(L_{b^\rho}^2\vee K_{\sigma^\rho}^2)T(T+4)}\eqqcolon c_6.
		\end{align*}
	\end{lem}
	\begin{proof}
		 Follows from \cite[Lemma~3.2 of Chapter~2]{Mao_2011} in connection with Lemma~\ref{lem:linear_growth}.
	\end{proof}
	
	\begin{lem}\label{lem:supremum_Xeta}
		Let
		\begin{align*}
			c_{11}&\coloneqq2(K_b^2\lvert\pi\rvert+K_\sigma^2).
		\end{align*}
		It holds that
		\begin{align*}
			\sup_{t\in[0,T]}\mathbb{E}\bigl[\lvert X^\eta_t-X^\eta_{\Pi(t)}\rvert^2\bigr]
			\leq c_{11}\lvert\pi\rvert.
		\end{align*}
	\end{lem}
	\begin{proof}
		Let $t\in[0,T]$. Inserting the definitions of $X^\eta_t$ and $X^\eta_{\Pi(t)}$ and applying Itô's isometry we obtain
		\begin{align*}
			\mathbb{E}\bigl[\lvert X^\eta_t-X^\eta_{\Pi(t)}\rvert^2\bigr]
			&\leq2\mathbb{E}\Biggl[\biggl\lvert\int_{\Pi(t)}^tb(s,X^\eta_s,\tilde{Y}^\eta_s,\rho(s,X^\eta_s))\,\mathrm{d}s\biggr\rvert^2\Biggr]+2\mathbb{E}\Biggl[\biggl\lvert\int_{\Pi(t)}^t\sigma(s,X^\eta_s,\tilde{Y}^\eta_s)\,\mathrm{d}W_s\biggr\rvert^2\Biggr]\\
			&\leq2(t-\Pi(t))^2K_b^2+2(t-\Pi(t))K_\sigma^2\\
			&\leq2(K_b^2\lvert\pi\rvert+K_\sigma^2)\lvert\pi\rvert.\qedhere
		\end{align*}
	\end{proof}
	
	\subsection{A-Posteriori Error of the Auxiliary FBSDE}
	At first, we derive an upper bound for the $L^2$-approximation error between $(\tilde{X}^\eta,Y^\eta,Z^\eta)$ and $(\tilde{X}^{\eta,\pi},Y^{\eta,\pi},Z^{\eta,\pi})$, that is
	\begin{align*}
		\mathcal{E}^\eta_\pi\coloneqq\sup_{t\in[0,T]}\mathbb{E}\bigl[\lvert\tilde{X}^\eta_t-\tilde{X}^{\eta,\pi}_t\rvert^2\bigr]+\max_{t_i\in\pi}\mathbb{E}\bigl[\lvert Y^\eta_{t_i}-Y^{\eta,\pi}_{t_i}\rvert^2\bigr]+\mathbb{E}\biggl[\int_0^T\lvert Y^\eta_t-Y^{\eta,\pi}_t\rvert^2+\lvert Z^\eta_t-Z^{\eta,\pi}_t\rvert^2\,\mathrm{d}t\biggr],
	\end{align*}
	which will rely on the a-posteriori estimates for BSDEs in \cite{Bender_Steiner_2013}. Additionally, we will use the standard result that, under the usual Lipschitz and linear growth conditions, the Euler-Maruyama scheme converges strongly with order $\frac{1}{2}$; see, e.\,g., \cite[Theorem~10.2.2]{Kloeden_Platen_1992}. To be able to indicate the dependence on $K_\rho$ and $H_\rho$ we give a self-contained proof of this result in Appendix~\ref{appendix_A}.
	
	\begin{lem}\label{lem:Euler_Scheme}
		Define
		\begin{align*}
			c_3&\coloneqq4\bigl(2T(TH_{b^\rho}^2+H_{\sigma^\rho}^2)+L_{b^\rho}^2\lvert\pi\rvert(1+c_6)+K_{\sigma^\rho}^2\bigr)\mathrm{e}^{8T(TH_{b^\rho}^2+H_{\sigma^\rho}^2)}.
		\end{align*}
		It holds that
		\begin{align*}
			\sup_{t\in[0,T]}\mathbb{E}\bigl[\lvert\tilde{X}^\eta_t-\tilde{X}^{\eta,\pi}_t\rvert^2\bigr]\leq c_3\lvert\pi\rvert.
		\end{align*}
	\end{lem}
	\begin{proof}
		See Appendix~\ref{appendix_A}.
	\end{proof}

	\begin{theorem}\label{theorem:estimate_rho_approx}
		Define
		\begin{align*}
			c_1&\coloneqq2(2T+1)(d_1+dc_7),\\
			c_4&\coloneqq\begin{multlined}[t]
				2d_2\bigl(H_g^2c_6+g(0)^2\bigr)+2(2T+1)\bigl(d_4H_g^2c_3+4d_5H_2^2(1+c_3)T\bigr)
				\\
				+4Td_3\bigl(H_F^2(T+2c_6)+F(0,0,0,0)^2+4K_{u_x}^2K_b^2+K_{u_{xx}}^2K_{\sigma\sigma^\top}^2\bigr),
			\end{multlined}\\
			\tilde{c}_4&\coloneqq8(2T+1)d_5H_3^2(1+c_3^\alpha)T,\\
			\Gamma&\coloneqq4H_1^2(T+1)(d\vee2)+16TH_1^4(1+T)^2(d\vee2)^2,
		\end{align*}
		where
		\begin{align*}
			c_7&\coloneqq\left(\frac{6}{(d\vee2)-1}+2+16\bigl(2+4(1+T)TH_1^2(d\vee2)\bigr)\right)\mathrm{e}^{\Gamma T},\\
			d_1&\coloneqq\left(\frac{d_4}{4}+2\right)(2+d)\bigl[2+c_7\bigl(1+H_1^2T(T+d)\bigr)\bigr],\\
			d_2&\coloneqq2H_1(2T+1)\tilde{d}_2+8(1+H_1^2T)\mathrm{e}^{(1+H_1)^2T}+16H_1^2\mathrm{e}^{(1+H_1)^2T}(T+2+2H_1^2T)\lvert\pi\rvert,\\
			d_3&\coloneqq2H_1(2T+1)\tilde{d}_3+8(1+H_1^2T)\mathrm{e}^{(1+H_1)^2T}+\bigl(16H_1^2\mathrm{e}^{(1+H_1)^2T}(3+2H_1^2T)+8\bigr)\lvert\pi\rvert,\\
			d_4&\coloneqq(1+2(1+2\beta_\pi)T)\mathrm{e}^{\alpha_\pi T}+4+4\beta_\pi\lvert\pi\rvert,\\
			d_5&\coloneqq\frac{d_4}{4H_1},\\
			\tilde{d}_2&\coloneqq\frac{d_4}{4H_1}\Bigl[(4+4H_1^2)\mathrm{e}^{(1+H_1)^2T}+8H_1^2\mathrm{e}^{(1+H_1)^2T}(T+2+2H_1^2T)\lvert\pi\rvert\Bigr],\\
			\tilde{d}_3&\coloneqq\frac{d_4}{4H_1}\Bigl[4(1+H_1^2T)\mathrm{e}^{(1+H_1)^2T}+\bigl(8H_1^2\mathrm{e}^{(1+H_1)^2T}(3+2H_1^2T)+4\bigr)\lvert\pi\rvert\Bigr],\\
			\alpha_\pi&\coloneqq\beta_\pi\left(1+\frac{1}{2(1-\lvert\pi\rvert/2)^2}\lvert\pi\rvert\right)+\frac{1}{2(1-\lvert\pi\rvert/2)^2},\\
			\beta_\pi&\coloneqq(4H_1+16H_1^2)\mathrm{e}^{(4H_1+16H_1^2)\lvert\pi\rvert}.
		\end{align*}
		For $\lvert\pi\rvert\leq\min\bigl\{1,(8H_1(1+4H_1))^{-1},\Gamma^{-1}\bigr\}$ it holds that
		\begin{align*}
			&\max_{t_i\in\pi}\mathbb{E}\bigl[\lvert Y^\eta_{t_i}-Y^{\eta,\pi}_{t_i}\rvert^2\bigr]+\mathbb{E}\left[\int_0^T\lvert Y^\eta_t-Y^{\eta,\pi}_t\rvert^2+\lvert Z^\eta_t-Z^{\eta,\pi}_t\rvert^2\,\mathrm{d}t\right]\\
			&\leq c_1\mathbb{E}\bigl[\lvert g(X^{\eta,\pi}_T)-Y^{\eta,\pi}_T\rvert^2\bigr]+c_4\lvert\pi\rvert+\tilde{c}_4\lvert\pi\rvert^\alpha.
		\end{align*}
	\end{theorem}
	
	\begin{proof}
		We define the random variables
		\begin{align*}
			\xi\coloneqq g(X^\eta_T)
			\quad\text{and}\quad
			\xi^\pi\coloneqq g(X^{\eta,\pi}_{t_N})=g(X^{\eta,\pi}_T),
		\end{align*}	
		as well as the random functions
		\begin{align*}
			f\colon\Omega\times[0,T]\times\mathbb{R}\times\mathbb{R}^d\to\mathbb{R}\quad\text{and}\quad f^\pi\colon\Omega\times\pi\times\mathbb{R}\times\mathbb{R}^d\to\mathbb{R},
		\end{align*}
		with
		\begin{align*}
			f(t,y,z)&\coloneqq f(\omega,t,y,z)\coloneqq F^\rho(t,\tilde{X}^\eta_t(\omega),y,z),\\
			f^\pi(t_i,y,z)&\coloneqq f^\pi(\omega,t_i,y,z)\coloneqq F^\rho(t_i,\tilde{X}^{\eta,\pi}_{t_i}(\omega),y,z).
		\end{align*}
		Since $X^\eta_T$ and $X^{\eta,\pi}_T$ are $\mathcal{F}_T$-measurable, from the measurability of $g$ it follows that $\xi$ and $\xi^\pi$ are $\mathcal{F}_T$-measurable. Furthermore, from the Lipschitz continuity of $g$ we obtain
		\begin{align*}
			\mathbb{E}\bigl[\lvert\xi\rvert^2\bigr]&=\mathbb{E}\bigl[\lvert g(X^\eta_T)-g(0)+g(0)\rvert^2\bigr]\leq2H_g^2\mathbb{E}\bigl[\lvert X^\eta_T\rvert^2\bigr]+2g(0)^2<\infty,\\
			\mathbb{E}\bigl[\lvert\xi^\pi\rvert^2\bigr]&=\mathbb{E}\bigl[\lvert g(X^{\eta,\pi}_T)-g(0)+g(0)\rvert^2\bigr]\leq2H_g^2\mathbb{E}\bigl[\lvert X^{\eta,\pi}_T\rvert^2\bigr]+2g(0)^2<\infty,
		\end{align*}	
		i.\,e., $\xi$ and $\xi^\pi$ are square-intgrable. Since $F^\rho$ is measurable (since it fulfills Hölder and Lipschitz conditions in all variables and is thus continuous) and $\tilde{X}^{\eta,\pi}_{t_i}$ is $\mathcal{F}_{t_i}$-measurable for all $t_i\in\pi$ it follows that $f^\pi$ is measurable and, for every $t_i\in\pi$ and $(y,z)\in\mathbb{R}\times\mathbb{R}^d$, $f^\pi(t_i,y,z)$ is $\mathcal{F}_{t_i}$-measurable. Also, since $F^\rho$ is measurable and $\tilde{X}^\eta$ is $(\mathcal{F}_t)_{t\geq0}$-progressively measurable, $f$ is measurable and, for every $(y,z)\in\mathbb{R}\times\mathbb{R}^d$, $f(\cdot,y,z)$ is $\mathcal{F}_t$-adapted. Moreover, from Lemma~\ref{lem:Lipschitz}(a) it directly follows that~$f$ and~$f^\pi$ are Lipschitz in~$(y,z)$ uniformly in~$(\omega,t)$ resp.\@~$(\omega,t_i)$ with Lipschitz constant $H_1$. Lastly, as in the proof of Theorem~\ref{theorem:auxFBSDE-sol}, it holds that
		\begin{align*}
			\mathbb{E}\left[\int_0^T\lvert f(t,0,0)\rvert^2\,\mathrm{d}t\right]
			&=\mathbb{E}\biggl[\int_0^T\lvert F^\rho(t,\tilde{X}^\eta_t,0,0)\,\mathrm{d}t\biggr]
			\overset{\eqref{eq:int_cond_X_eta_tilde}}{<}\infty,
		\end{align*}
		and, similarly, 
		\begin{align*}
			\mathbb{E}\bigl[\lvert f^\pi(t_i,0,0)\rvert^2\bigr]
			&=\mathbb{E}\bigl[\lvert F^\rho(t_i,\tilde{X}^{\eta,\pi}_{t_i},0,0)-F^\rho(t_i,0,0,0)+F^\rho(t_i,0,0,0)\rvert^2\bigr]\\
			&\leq8\mathbb{E}\left[H_2^2t_i+H_2^2\lvert\tilde{X}^{\eta,\pi}_{t_i}\rvert^2+H_3^2t_i^\alpha+H_3^2\lvert\tilde{X}^{\eta,\pi}_{t_i}\rvert^{2\alpha}\right]+2F^\rho(0,0,0,0)^2\\
			&<\infty.
		\end{align*}
		Thus, we can apply \cite[Theorem~3.1]{Bender_Steiner_2013}, which yields
		\begingroup
		\thinmuskip 0.75mu
		\medmuskip  1mu plus 0.5mu minus 1mu
		\thickmuskip  1.25mu plus 1.25mu
		\begin{align*}
			&\max_{t_i\in\pi}\mathbb{E}\bigl[\lvert Y^\eta_{t_i}-Y^{\eta,\pi}_{t_i}\rvert^2\bigr]+\mathbb{E}\left[\int_0^T\lvert Y^\eta_t-Y^{\eta,\pi}_t\rvert^2+\lvert Z^\eta_t-Z^{\eta,\pi}_t\rvert^2\,\mathrm{d}t\right]\\
			&\begin{multlined}[t]\leq
				2(2T+1)(d_1+dc_7)\mathcal{E}+\left(d_2\mathbb{E}\bigl[\lvert\xi\rvert^2\bigr]+d_3\mathbb{E}\left[\int_0^T\lvert F^\rho(r,\tilde{X}^\eta_r,0,0)\rvert^2\,\mathrm{d}r\right]\right)\lvert\pi\rvert\\
				+2(2T+1)d_4\mathbb{E}\bigl[\lvert\xi-\xi^\pi\rvert^2\bigr]+2(2T+1)d_5\int_0^T\mathbb{E}\bigl[\lvert F^\rho(t,\tilde{X}^\eta_t,Y^\eta_t,Z^\eta_t)-F^\rho(\Pi(t),\tilde{X}^{\eta,\pi}_t,Y^\eta_t,Z^\eta_t)\rvert^2\bigr]\,\mathrm{d}t
			\end{multlined}\\
			&=2(2T+1)(d_1+c_7)\mathcal{E}+(d_2(\mathrm{II})+d_3(\mathrm{III}))\lvert\pi\rvert+2(2T+1)d_4(\mathrm{IV})+2(2T+1)d_5(\mathrm{V}),
		\end{align*}
		\endgroup
		where
		\begin{align*}
			\mathcal{E}\coloneqq\mathbb{E}\bigl[\lvert\xi^\pi-Y^{\eta,\pi}_{t_N}\rvert^2\bigr]+\max_{1\leq i\leq N}\mathbb{E}\Biggl[\biggl\lvert Y^{\eta,\pi}_{t_i}-Y^{\eta,\pi}_{t_0}-\sum_{j=0}^{i-1}\bigl(f^\pi(t_j,Y^{\eta,\pi}_{t_j},Z^{\eta,\pi}_{t_j})\Delta_j+Z^{\eta,\pi}_{t_j}\Delta W_j\bigr)\biggr\rvert^2\Biggr].
		\end{align*}
		
		Observe that
		\begin{align*}
			f^\pi(t_i,Y^{\eta,\pi}_{t_i},Z^{\eta,\pi}_{t_i})
			&=F^\rho(t_i,\tilde{X}^{\eta,\pi}_{t_i},Y^{\eta,\pi}_{t_i},Z^{\eta,\pi}_{t_i})
			\overset{\eqref{eq:F_rho_for_approximation_triple}}{=}F(t_i,X^{\eta,\pi}_{t_i},Y^{\eta,\pi}_{t_i},\rho(t_i,X^{\eta,\pi}_{t_i})).
		\end{align*}	
		Consequently, since, by the recursion of the Euler-Maruyama scheme, it holds that
		\begin{align*}
			Y^{\eta,\pi}_{t_i}&=Y^{\eta,\pi}_{t_0}+\sum_{j=0}^{i-1}\bigl(F(t_j,X^{\eta,\pi}_{t_j},Y^{\eta,\pi}_{t_j},\rho(t_j,X^{\eta,\pi}_{t_j}))\Delta_j+\rho(t_j,X^{\eta,\pi}_{t_j})\Delta W_j\bigr),
		\end{align*}
		the term $\mathcal{E}$ simplifies to
		\begin{align*}
			\mathcal{E}=\mathbb{E}\bigl[\lvert\xi^\pi-Y^{\eta,\pi}_{t_N}\rvert^2\bigr].
		\end{align*}
		
		Inserting the definition $\xi=g(X^\eta_T)$ and using the Lipschitz continuity of $g$, we get that
		\begin{align*}
			(\mathrm{II})
			=\mathbb{E}\bigl[\lvert g(X^\eta_T)-g(0)+g(0)\rvert^2\bigr]
			\leq2\bigl(H_g^2\mathbb{E}\bigl[\lvert X^\eta_T\rvert^2\bigr]+g(0)^2\bigr)
			\leq2\bigl(H_g^2c_6+g(0)^2\bigr).
		\end{align*}
		
		Similarly, now using the Hölder and Lipschitz conditions on $F$, it follows that
		\begin{align*}
			(\mathrm{III})
			&\leq4\mathbb{E}\biggl[\int_0^T
				\begin{multlined}[t]
					\lvert F(r,X^\eta_r,0,0)-F(0,0,0,0)\rvert^2+F(0,0,0,0)^2\\
					+\bigl\lvert u_x(r,X^\eta_r)\bigl[b(r,X^\eta_r,\tilde{Y}^\eta_r,\rho(r,X^\eta_r))-b(r,X^\eta_r,0,0)\bigr]\bigr\rvert^2\\
					+\frac{1}{4}\mathrm{tr}\bigl[u_{xx}(r,X^\eta_r)\bigl((\sigma\sigma^\top)(r,X^\eta_r,\tilde{Y}^\eta_r)-(\sigma\sigma^\top)(r,X^\eta_r,0)\bigr)\bigr]^2\,\mathrm{d}r\biggr]
				\end{multlined}\\
			&\leq4H_F^2\left(2\int_0^Tr+\mathbb{E}\bigl[\lvert X^\eta_r\rvert^2\bigr]\,\mathrm{d}r\right)+4T\bigl(F(0,0,0,0)^2+4K_{u_x}^2K_b^2+K_{u_{xx}}^2K_{\sigma\sigma^\top}^2\bigr)\\
			&\leq4H_F^2\biggl(T^2+2T\mathbb{E}\biggl[\sup_{t\in[0,T]}\lvert X^\eta_t\rvert^2\biggr]\biggr)+4T\bigl(F(0,0,0,0)^2+4K_{u_x}^2K_b^2+K_{u_{xx}}^2K_{\sigma\sigma^\top}^2\bigr)\\
			&\leq4T\bigl(H_F^2(T+2c_6)+F(0,0,0,0)^2+4K_{u_x}^2K_b^2+K_{u_{xx}}^2K_{\sigma\sigma^\top}^2\bigr).
		\end{align*}
		
		Inserting the definitions $\xi=g(X^\eta_T)$ and $\xi^\pi=g(X^{\eta,\pi}_T)$, using the Lipschitz continuity of $g$, and applying Lemma~\ref{lem:Euler_Scheme}, we obtain
		\begin{align*}
			(\mathrm{IV})
			\leq H_g^2\mathbb{E}\bigl[\lvert X^\eta_T-X^{\eta,\pi}_T\rvert^2\bigr]
			\leq H_g^2\mathbb{E}\bigl[\lvert\tilde{X}^\eta_T-\tilde{X}^{\eta,\pi}_T\rvert^2\bigr]
			\leq H_g^2c_3\lvert\pi\rvert.
		\end{align*}
		From Lemma~\ref{lem:Lipschitz}(b), Lemma~\ref{lem:Euler_Scheme} and Jensen's inequality we obtain
		\begin{align*}
			(\mathrm{V})
			&\leq\begin{multlined}[t]
				2H_2^2\left(\int_0^T\mathbb{E}\Bigl[\bigl(\lvert t-\Pi(t)\rvert^{1/2}+\lvert\tilde{X}^\eta_t-\tilde{X}^{\eta,\pi}_t\rvert\bigr)^2\Bigr]\,\mathrm{d}t\right)\\
				+2H_3^2\left(\int_0^T\mathbb{E}\Bigl[\bigl(\lvert t-\Pi(t)\rvert^{\alpha/2}+\lvert\tilde{X}^\eta_t-\tilde{X}^{\eta,\pi}_t\rvert^\alpha\bigr)^2\Bigr]\,\mathrm{d}t\right)
			\end{multlined}\\
			&\leq\begin{multlined}[t]
				4H_2^2\biggl(\int_0^T\mathbb{E}\bigl[\lvert t-\Pi(t)\rvert\bigr]+\sup_{s\in[0,T]}\mathbb{E}\bigl[\lvert\tilde{X}^\eta_s-\tilde{X}^{\eta,\pi}_s\rvert^2\bigr]\,\mathrm{d}t\biggr)\\
				+4H_3^2\biggl(\int_0^T\mathbb{E}\bigl[\lvert t-\Pi(t)\rvert^\alpha\bigr]+\sup_{s\in[0,T]}\mathbb{E}\bigl[\lvert\tilde{X}^\eta_s-\tilde{X}^{\eta,\pi}_s\rvert^{2\alpha}\bigr]\,\mathrm{d}t\biggr)
			\end{multlined}\\
			&\leq4H_2^2(1+c_3)T\lvert\pi\rvert+4H_3^2\biggl(\int_0^T\mathbb{E}\bigl[\lvert t-\Pi(t)\rvert^\alpha\bigr]+\biggl(\sup_{s\in[0,T]}\mathbb{E}\bigl[\lvert\tilde{X}^\eta_s-\tilde{X}^{\eta,\pi}_s\rvert^{2}\bigr]\biggr)^\alpha\,\mathrm{d}t\biggr)\\
			&\leq4H_2^2(1+c_3)T\lvert\pi\rvert+4H_3^2(1+c_3^\alpha)T\lvert\pi\rvert^\alpha.
		\end{align*}
		Putting the above estimates together gives the result.
	\end{proof}
	
	\subsection{Connecting the Original FBSDE and the Auxiliary Decoupled FBSDE}
	In a second step, we estimate the deviation between the adapted solution $(X,Y,Z)$ of the original FBSDE~\eqref{eq:FBSDE} and the adapted solution $(X^\eta,Y^\eta,Z^\eta)$ of the modified FBSDE~\eqref{eq:FBSDE-rho}:
	\begin{align*}
		\mathcal{E}_\eta\coloneqq\sup_{t\in[0,T]}\mathbb{E}\bigl[\lvert X_t-X^\eta_t\rvert^2\bigr]+\max_{t_i\in\pi}\mathbb{E}\bigl[\lvert Y_{t_i}-Y^\eta_{t_i}\rvert^2\bigr]+\mathbb{E}\left[\int_0^T\bigl(\lvert Y_t-Y^\eta_t\rvert^2+\lvert Z_t-Z^\eta_t\rvert^2\bigr)\,\mathrm{d}t\right].
	\end{align*}
	
	\begin{lem}\label{lem:estimate_SDE}
		Let
		\begin{align*}
			c_8&\coloneqq\Bigl[12H_b^2\bigl(1+(1+H_\vartheta)^2T\bigr)T+10H_\sigma^2T\Bigr]c_1\mathrm{e}^{c_{10}},\\
			c_9&\coloneqq\begin{multlined}[t]
				\Bigl[12H_b^2T^2\bigl((H_u+H_\rho+H_\vartheta H_u)^2+H_u^2(1+H_\vartheta)^2c_{11}+(1+H_\rho+H_\vartheta)^2c_3\bigr)\\
				+10H_\sigma^2T(H_u^2+c_3+H_u^2c_{11})+\Bigl(12H_b^2\bigl(1+(1+H_\vartheta)^2T\bigr)T+10H_\sigma^2T\Bigr)c_4\Bigr]\mathrm{e}^{c_{10}},
			\end{multlined}\\
			\tilde{c}_9&\coloneqq\Bigl[12H_b^2\bigl(1+(1+H_\vartheta)^2T\bigr)T+10H_\sigma^2T\Bigr]\tilde{c}_4\mathrm{e}^{c_{10}},
		\end{align*}
		where
		\begin{align*}
		c_{10}&\coloneqq(1+H_u)^2T\bigl(12H_b^2(1+H_\vartheta)^2T+10H_\sigma^2\bigr).
		\end{align*}
		It holds that
		\begin{align*}
			\sup_{t\in[0,T]}\mathbb{E}\bigl[\lvert X_t-X^\eta_t\rvert^2\bigr]
			\leq c_8\mathbb{E}\bigl[\lvert\xi^\pi-Y^{\eta,\pi}_{t_N}\rvert^2\bigr]+c_9\lvert\pi\rvert+\tilde{c}_9\lvert\pi\rvert^\alpha.
		\end{align*}
	\end{lem}
	\begin{proof}
		From the definitions of $X_t$ and $X^\eta_t$, Hölder's inequality and Itô's isometry in connection with Fubini's theorem we obtain
		\begin{multline*}
			\mathbb{E}\bigl[\lvert X_t-X^\eta_t\rvert^2\bigr]
			\leq2T\int_0^t\mathbb{E}\bigl[\lvert b(s,X_s,Y_s,Z_s)-b(s,X^\eta_s,\tilde{Y}^\eta_s,\rho(s,X^\eta_s))\rvert^2\bigr]\,\mathrm{d}s\\
			+2\int_0^t\mathbb{E}\bigl[\lvert\sigma(s,X_s,Y_s)-\sigma(s,X^\eta_s,\tilde{Y}^\eta_s)\rvert^2\bigr]\,\mathrm{d}s.
		\end{multline*}
		Note that, using the identities of Theorem~\ref{theorem:auxFBSDE-sol}, we obtain
		\begin{align}\label{eq:diff_Y_Yrhotilde}
			&\lvert Y_s-\tilde{Y}^\eta_s\rvert\nonumber\\
			&\leq\lvert Y_s-u(s,X^\eta_s)\rvert+\bigl\lvert u(s,X^\eta_s)-u\bigl(\Pi(s),X^\eta_{\Pi(s)}\bigr)\bigr\rvert+\bigl\lvert u\bigl(\Pi(s),X^\eta_{\Pi(s)}\bigr)-Y^{\eta,\pi}_{\Pi(s)}\bigr\rvert+\bigl\lvert Y^{\eta,\pi}_{\Pi(s)}-\tilde{Y}^\eta_s\bigr\rvert\nonumber\\
			&=\lvert u(s,X_s)-u(s,X^\eta_s)\rvert+\bigl\lvert u(s,X^\eta_s)-u\bigl(\Pi(s),X^\eta_{\Pi(s)}\bigr)\bigr\rvert+\bigl\lvert Y^\eta_{\Pi(s)}-Y^{\eta,\pi}_{\Pi(s)}\bigr\rvert+\lvert\tilde{Y}^{\eta,\pi}_s-\tilde{Y}^\eta_s\rvert\nonumber\\
			&\leq H_u\lvert X_s-X^\eta_s\rvert+H_u\lvert s-\Pi(s)\rvert^{1/2}+H_u\bigl\lvert X^\eta_s-X^\eta_{\Pi(s)}\bigr\rvert+\bigl\lvert Y^\eta_{\Pi(s)}-Y^{\eta,\pi}_{\Pi(s)}\bigr\rvert+\lvert\tilde{X}^{\eta,\pi}_s-\tilde{X}^\eta_s\rvert.
		\end{align}
		Thus, together with Remark~\ref{rem:HLC_uxsigma}, we can bound
		\begin{align}\label{eq:diff_Z_rho}
			&\lvert Z_s-\rho(s,X^\eta_s)\rvert\nonumber\\
			&\overset{\hphantom{\eqref{eq:diff_Y_Yrhotilde}}}{\leq}\lvert Z_s-\vartheta(s,X^\eta_s,\tilde{Y}^\eta_s)\rvert+\bigl\lvert \vartheta(s,X^\eta_s,\tilde{Y}^\eta_s)-\rho\bigl(\Pi(s),X^{\eta,\pi}_{\Pi(s)}\bigr)\bigr\rvert+\bigl\lvert\rho\bigl(\Pi(s),X^{\eta,\pi}_{\Pi(s)}\bigr)-\rho(s,X^\eta_s)\bigr\rvert\nonumber\\
			&\overset{\hphantom{\eqref{eq:diff_Y_Yrhotilde}}}{=}\begin{multlined}[t]
				\lvert\vartheta(s,X_s,Y_s)-\vartheta(s,X^\eta_s,\tilde{Y}^\eta_s)\rvert+\lvert Z^\eta_s-Z^{\eta,\pi}_s\rvert+\lvert\rho\left(\Pi(s),X^{\eta,\pi}_s\right)-\rho(s,X^\eta_s)\rvert
			\end{multlined}\nonumber\\
			&\overset{\hphantom{\eqref{eq:diff_Y_Yrhotilde}}}{\leq}H_\vartheta\lvert X_s-X^\eta_s\rvert+H_\vartheta\lvert Y_s-\tilde{Y}^\eta_s\rvert+\lvert Z^\eta_s-Z^{\eta,\pi}_s\rvert+H_\rho\lvert s-\Pi(s)\rvert^{1/2}+H_\rho\lvert X^\eta_s-X^{\eta,\pi}_s\rvert\nonumber\\
			&\begin{multlined}[b]\overset{\eqref{eq:diff_Y_Yrhotilde}}{\leq}
				H_\vartheta(1+H_u)\lvert X_s-X^\eta_s\rvert+\lvert Z^\eta_s-Z^{\eta,\pi}_s\rvert+(H_\rho+H_\vartheta H_u)\lvert s-\Pi(s)\rvert^{1/2}\\
				\qquad+(H_\rho+H_\vartheta)\lvert\tilde{X}^\eta_s-\tilde{X}^{\eta,\pi}_s\rvert+H_\vartheta H_u\bigl\lvert X^\eta_s-X^\eta_{\Pi(s)}\bigr\rvert+H_\vartheta\bigl\lvert Y^\eta_{\Pi(s)}-Y^{\eta,\pi}_{\Pi(s)}\bigr\rvert.
			\end{multlined}
		\end{align}
		Then, exploiting the $\frac{1}{2}$-Hölder and Lipschitz continuity of $b$, $\sigma$, $u$, $u_x$ and $\rho$ as well as the boundedness of~$\sigma$ and~$u_x$, it follows that
		\begingroup
		\thinmuskip 0.75mu
		\medmuskip  1mu plus 0.5mu minus 1mu
		\thickmuskip  1.25mu plus 1.25mu
		\begin{align*}
			&\int_0^t\mathbb{E}\bigl[\lvert b(s,X_s,Y_s,Z_s)-b(s,X^\eta_s,\tilde{Y}^\eta_s,\rho(s,X^\eta_s))\rvert^2\bigr]\,\mathrm{d}s\\
			&\overset{\hphantom{\eqref{eq:diff_Y_Yrhotilde},\;\eqref{eq:diff_Z_rho}}}{\leq}H_b^2\int_0^t\mathbb{E}\Bigl[\bigl(\lvert X_s-X^\eta_s\rvert+\lvert Y_s-\tilde{Y}^\eta_s\rvert+\lvert Z_s-\rho(s,X^\eta_s)\rvert\bigr)^2\Bigr]\,\mathrm{d}s\\
			&\overset{\eqref{eq:diff_Y_Yrhotilde},\;\eqref{eq:diff_Z_rho}}{\leq}
				\begin{multlined}[t][0.9\displaywidth]
					H_b^2\int_0^t\mathbb{E}\Bigl[\Bigl((1+H_u)(1+H_\vartheta)\lvert X_s-X^\eta_s\rvert+\lvert Z^\eta_s-Z^{\eta,\pi}_s\rvert+(H_u+H_\rho+H_\vartheta H_u)\lvert s-\Pi(s)\rvert^{1/2}\\
				+(1+H_\rho+H_\vartheta)\lvert\tilde{X}^\eta_s-\tilde{X}^{\eta,\pi}_s\rvert+H_u(1+H_\vartheta)\bigl\lvert X^\eta_s-X^\eta_{\Pi(s)}\bigr\rvert+(1+H_\vartheta)\bigl\lvert Y^\eta_{\Pi(s)}-Y^{\eta,\pi}_{\Pi(s)}\bigr\rvert\Bigr)^2\Bigr]\,\mathrm{d}s
				\end{multlined}\\
			&\overset{\hphantom{\eqref{eq:diff_Y_Yrhotilde},\;\eqref{eq:diff_Z_rho}}}{\leq}
				\begin{multlined}[t][0.9\displaywidth]
					6H_b^2(1+H_u)^2(1+H_\vartheta)^2\int_0^t\mathbb{E}\bigl[\lvert X_s-X^\eta_s\rvert^2\bigr]\,\mathrm{d}s\\
					+6H_b^2\left(\mathbb{E}\left[\int_0^T\lvert Z^\eta_s-Z^{\eta,\pi}_s\rvert^2\right]\,\mathrm{d}s+(1+H_\vartheta)^2T\max_{t_i\in\pi}\mathbb{E}\bigl[\lvert Y^\eta_{t_i}-Y^{\eta,\pi}_{t_i}\rvert^2\bigr]\right)\\
					+6H_b^2\int_0^T(H_u+H_\rho+H_\vartheta H_u)^2\lvert\pi\rvert+H_u^2(1+H_\vartheta)^2\sup_{r\in[0,T]}\mathbb{E}\bigl[\lvert X^\eta_r-X^\eta_{\Pi(r)}\rvert^2\bigr]\\
					+(1+H_\rho+H_\vartheta)^2\sup_{r\in[0,T]}\mathbb{E}\bigl[\lvert\tilde{X}^\eta_r-\tilde{X}^{\eta,\pi}_r\rvert^2\bigr]\,\mathrm{d}s
				\end{multlined}\\
			&\overset{\hphantom{\eqref{eq:diff_Y_Yrhotilde},\;\eqref{eq:diff_Z_rho}}}{\leq}
				\begin{multlined}[t]
					6H_b^2(1+H_u)^2(1+H_\vartheta)^2\int_0^t\mathbb{E}\bigl[\lvert X_s-X^\eta_s\rvert^2\bigr]\,\mathrm{d}s\\
					\shoveleft{+6H_b^2(1+(1+H_\vartheta)^2T)\Bigl(c_1\mathbb{E}\bigl[\lvert\xi^\pi-Y^{\eta,\pi}_{t_N}\rvert^2\bigr]+c_4\lvert\pi\rvert+\tilde{c}_4\lvert\pi\rvert^\alpha}\Bigr)\\
					+6H_b^2T\bigl((H_u+H_\rho+H_\vartheta H_u)^2+H_u^2(1+H_\vartheta)^2c_{11}
					+(1+H_\rho+H_\vartheta)^2c_3\bigr)\lvert\pi\rvert.
				\end{multlined}
		\end{align*}
		\endgroup
		To derive the last inequality we additionally made use of Theorem~\ref{theorem:estimate_rho_approx} and Lemmas~\ref{lem:supremum_Xeta} and~\ref{lem:Euler_Scheme}. Moreover, under the Lipschitz continuity of $\sigma$ it holds that
		\begin{align*}
			&\int_0^t\mathbb{E}\bigl[\lvert\sigma(s,X_s,Y_s)-\sigma(s,X^\eta_s,\tilde{Y}^\eta_s)\rvert^2\bigl]\,\mathrm{d}s\\
			&\overset{\hphantom{\eqref{eq:diff_Y_Yrhotilde}}}{\leq}H_\sigma^2\int_0^t\mathbb{E}\Bigl[\bigl(\lvert X_s-X^\eta_s\rvert+\lvert Y_s-\tilde{Y}^\eta_s\rvert\bigr)^2\Bigr]\,\mathrm{d}s\\
			&\overset{\eqref{eq:diff_Y_Yrhotilde}}{\leq}5H_\sigma^2\begin{multlined}[t]
				\int_0^t(1+H_u)^2\mathbb{E}\bigl[\lvert X_s-X^\eta_s\rvert^2\bigr]+H_u^2\mathbb{E}[\lvert s-\Pi(s)\rvert]+H_u^2\mathbb{E}\bigl[\lvert X^\eta_s-X^\eta_{\Pi(s)}\rvert^2\bigr]\\
				+\mathbb{E}\bigl[\lvert Y^\eta_{\Pi(s)}-Y^{\eta,\pi}_{\Pi(s)}\rvert^2\bigr]+\mathbb{E}\bigl[\lvert\tilde{X}^{\eta,\pi}_s-\tilde{X}^\eta_s\rvert^2\bigr]\,\mathrm{d}s
			\end{multlined}\\
			&\overset{\hphantom{\eqref{eq:diff_Y_Yrhotilde}}}{\leq}\begin{multlined}[t]
				5H_\sigma^2(1+H_u)^2\int_0^t\mathbb{E}\bigl[\lvert X_s-X^\eta_s\rvert^2\bigr]\,\mathrm{d}s
				+5H_\sigma^2\int_0^tH_u^2\lvert\pi\rvert+H_u^2\sup_{r\in[0,T]}\mathbb{E}\bigl[\lvert X^\eta_r-X^\eta_{\Pi(r)}\rvert^2\bigr]\\
				+\max_{t_i\in\pi}\mathbb{E}\bigl[\lvert Y^\eta_{t_i}-Y^{\eta,\pi}_{t_i}\rvert^2\bigr]+\sup_{r\in[0,T]}\mathbb{E}\bigl[\lvert\tilde{X}^{\eta,\pi}_r-\tilde{X}^\eta_r\rvert^2\bigr]\,\mathrm{d}s
			\end{multlined}\\
			&\overset{\hphantom{\eqref{eq:diff_Y_Yrhotilde}}}{\leq}\begin{multlined}[t]
				5H_\sigma^2(1+H_u)^2\int_0^t\mathbb{E}\bigl[\lvert X_s-X^\eta_s\rvert^2\bigr]\\
				+5H_\sigma^2T\bigl(c_1\mathbb{E}\bigl[\lvert\xi^\pi-Y^{\eta,\pi}_{t_N}\rvert^2\bigr]+c_4\lvert\pi\rvert+\tilde{c}_4\lvert\pi\rvert^\alpha\bigr)+5H_\sigma^2T(H_u^2+c_3+H_u^2c_{11})\lvert\pi\rvert.
			\end{multlined}
		\end{align*}
		Together we obtain
		\begin{align*}
			\mathbb{E}\bigl[\lvert X_t-X^\eta_t\rvert^2\bigr]
			&\leq\begin{multlined}[t]
				\bigl[12H_b^2(1+H_u)^2(1+H_\vartheta)^2T+10H_\sigma^2(1+H_u)^2\bigr]\int_0^t\mathbb{E}\bigl[\lvert X_s-X^\eta_s\rvert^2\bigr]\,\mathrm{d}s\\
				\shoveleft{+\bigl[12H_b^2\bigl(1+(1+H_\vartheta)^2T\bigr)T+10H_\sigma^2T\bigr]\bigl(c_1\mathbb{E}\bigl[\lvert\xi^\pi-Y^{\eta,\pi}_{t_N}\rvert^2\bigr]+c_4\lvert\pi\rvert+\tilde{c}_4\lvert\pi\rvert^\alpha\bigr)}\\
				\shoveleft{+\bigl[12H_b^2T^2\bigl((H_u+H_\rho+H_\vartheta H_u)^2+H_u^2(1+H_\vartheta)^2c_{11}
					+(1+H_\rho+H_\vartheta)^2c_3\bigr)}\\
				+10H_\sigma^2T(H_u^2+c_3+H_u^2c_{11})\bigr]\lvert\pi\rvert.
			\end{multlined}
		\end{align*}
		Applying Gronwall's lemma finishes the proof.
	\end{proof}
	
	\begin{remark}
		From Equation~\eqref{eq:diff_Y_Yrhotilde} it follows that
		\begin{align*}
			\mathbb{E}\left[\int_0^t\lvert Y_s-\tilde{Y}^\eta_s\rvert^2\,\mathrm{d}s\right]
			&=\int_0^t\mathbb{E}\bigl[\lvert Y_s-\tilde{Y}^\eta_s\rvert^2\bigr]\,\mathrm{d}s\\
			&\leq5\int_0^t\begin{multlined}[t]
				H_u^2\mathbb{E}\bigl[\lvert X_s-X^\eta_s\rvert^2\bigr]+H_u^2\mathbb{E}\bigl[\lvert s-\Pi(s)\rvert\bigr]+H_u^2\mathbb{E}\bigl[\lvert X^\eta_s-X^\eta_{\Pi(s)}\rvert^2\bigr]\\
				+\mathbb{E}\bigl[\lvert Y^\eta_{\Pi(s)}-Y^{\eta,\pi}_{\Pi(s)}\rvert^2\bigr]+\mathbb{E}\bigl[\lvert\tilde{X}^{\eta,\pi}_s-\tilde{X}^\eta_s\rvert^2\bigr]\,\mathrm{d}s
			\end{multlined}\\
			&\leq5\int_0^t\begin{multlined}[t]
				H_u^2\sup_{r\in[0,T]}\mathbb{E}\bigl[\lvert X_r-X^\eta_r\rvert^2\bigr]+H_u^2\lvert\pi\rvert+H_u^2\sup_{r\in[0,T]}\mathbb{E}\bigl[\lvert X^\eta_r-X^\eta_{\Pi(r)}\rvert^2\bigr]\\
				+\max_{t_i\in\pi}\mathbb{E}\bigl[\lvert Y^\eta_{t_i}-Y^{\eta,\pi}_{t_i}\rvert^2\bigr]+\sup_{r\in[0,T]}\mathbb{E}\bigl[\lvert \tilde{X}^{\eta,\pi}_r-\tilde{X}^\eta_r\rvert^2\bigr]\,\mathrm{d}s
			\end{multlined}\\
			&\leq\begin{multlined}[t]
				5T(c_1+H_u^2c_8)\mathbb{E}\bigl[\lvert\xi^\pi-Y^{\eta,\pi}_{t_N}\rvert^2\bigr]+5T(\tilde{c}_4+H_u^2\tilde{c}_9)\lvert\pi\rvert^\alpha\\
				+5T\bigl(c_3+c_4+H_u^2(1+c_9+c_{11})\bigr)\lvert\pi\rvert.
			\end{multlined}
		\end{align*}
	\end{remark}
	
	\begin{theorem}\label{theorem:estimate_orig_rho}
		Define
		\begin{align*}
			c_2&\coloneqq\left(1+H_u^2(1+T)+2TH_\vartheta^2(1+5TH_u^2)\right)c_8+10T^2H_\vartheta c_1,\\
			c_5&\coloneqq\left(1+H_u^2(1+T)+2TH_\vartheta^2(1+5TH_u^2)\right)c_9+10T^2H_\vartheta\left(c_4+c_3+H_u^2(1+c_{11})\right),\\
			\tilde{c}_5&\coloneqq\left(1+H_u^2(1+T)+2TH_\vartheta^2(1+5TH_u^2)\right)\tilde{c}_9+10T^2H_\vartheta\tilde{c}_4.
		\end{align*}
		It holds that
		\begin{align*}
			\mathcal{E}_\eta\leq c_2\mathbb{E}\bigl[\lvert\xi^\pi-Y^{\eta,\pi}_{t_N}\rvert^2\bigr]+c_5\lvert\pi\rvert+\tilde{c}_5\lvert\pi\rvert^\alpha.
		\end{align*}
	\end{theorem}
	\begin{proof}
		Since $Y_t=u(t,X_t)$ and $Y^\eta_t=u(t,X^\eta_t)$, from the Lipschitz continuity of $u(t,x)$ in $x$ uniformly in $t$ and applying Lemma~\ref{lem:estimate_SDE} it follows that
		\begin{align*}
			\max_{t_i\in\pi}\mathbb{E}\bigl[\lvert Y_{t_i}-Y^\eta_{t_i}\rvert^2\bigr]
			&=\max_{t_i\in\pi}\mathbb{E}\bigl[\lvert u(t_i,X_{t_i})-u(t_i,X^\eta_{t_i})\rvert^2\bigr]\\
			&\leq H_u^2\max_{t_i\in\pi}\mathbb{E}\bigl[\lvert X_{t_i}-X^\eta_{t_i}\rvert^2\bigr]\\
			&\leq H_u^2c_8\mathbb{E}\bigl[\lvert\xi^\pi-Y^{\eta,\pi}_{t_N}\rvert^2\bigr]+H_u^2c_9\lvert\pi\rvert+H_u^2\tilde{c}_9\lvert\pi\rvert^\alpha.
		\end{align*}
		Using the same arguments as above and applying Fubini's theorem, we also obtain
		\begin{align*}
			\mathbb{E}\left[\int_0^T\lvert Y_t-Y^\eta_t\rvert^2\,\mathrm{d}t\right]
			&\leq H_u^2Tc_8\mathbb{E}\bigl[\lvert\xi^\pi-Y^{\eta,\pi}_{t_N}\rvert^2\bigr]+H_u^2Tc_9\lvert\pi\rvert+H_u^2T\tilde{c}_9\lvert\pi\rvert^\alpha.
		\end{align*}
		Similarly, now using the definitions $Z_t=v(t,X_t)=\vartheta(t,X_t,Y_t)$ and $Z^\eta_t=\vartheta(t,X^\eta_t,\tilde{Y}^\eta_t)$ as well as the Lipschitz continuity of $\vartheta(t,x,y)$ in $(x,y)$ uniformly in $t$ and the last remark, we get
		\begin{align*}
			\mathbb{E}\left[\int_0^T\lvert Z_t-Z^\eta_t\rvert^2\,\mathrm{d}t\right]
			&=\int_0^T\mathbb{E}\bigl[\lvert\vartheta(t,X_t,Y_t)-\vartheta(t,X^\eta_t,\tilde{Y}^\eta_t)\rvert^2\bigr]\,\mathrm{d}t\\
			&\leq2H_\vartheta^2\int_0^T\mathbb{E}\bigl[\lvert X_t-X^\eta_t\rvert^2\bigr]+\mathbb{E}\bigl[\lvert Y_t-\tilde{Y}^\eta_t\rvert^2\bigr]\,\mathrm{d}t\\
			&\leq\begin{multlined}[t]
				2TH_\vartheta^2\bigl(c_8+5T(c_1+H_u^2c_8)\bigr)\mathbb{E}\bigl[\lvert\xi^\pi-Y^{\eta,\pi}_{t_N}\rvert^2\bigr]\\
				+2TH_\vartheta^2\bigl(\tilde{c}_9+5T(\tilde{c}_4+H_u^2\tilde{c}_9)\bigr)\lvert\pi\rvert^\alpha\\
				+2TH_\vartheta^2\bigl[c_9+5T\bigl(c_3+c_4+H_u^2(1+c_9+c_{11})\bigr)\bigr]\lvert\pi\rvert.
			\end{multlined}
		\end{align*}
		Summing up the three estimates and again applying Lemma~\ref{lem:estimate_SDE} completes the proof.
	\end{proof}
	
	\begin{proof}[Proof of Theorem~\ref{theorem:main}]
		It holds that
		\begin{align*}
			\mathcal{E}_{\eta,\pi}&\leq2\left(\mathcal{E}_\eta+\mathcal{E}^\eta_\pi\right).
		\end{align*}
		Thus, from Lemma~\ref{lem:Euler_Scheme}, Theorem~\ref{theorem:estimate_rho_approx} and Theorem~\ref{theorem:estimate_orig_rho}, it follows that
		\begin{align}\label{eq:main_concrete}
			\mathcal{E}_{\eta,\pi}\leq2(c_1+c_2)\mathbb{E}\bigl[\lvert Y^{\eta,\pi}_T-g(X^{\eta,\pi}_T)\rvert^2\bigr]+2(c_3+c_4+c_5)\lvert\pi\rvert+2(\tilde{c}_4+\tilde{c}_5)\lvert\pi\rvert^\alpha,
		\end{align}
		for $\lvert\pi\rvert\leq\min\bigl\{1,(8H_1(1+4H_1))^{-1},\Gamma^{-1}\bigr\}$.
		The statement now follows from some basic upper bounding of the constants in dependence of $K_\rho$ and $H_\rho$ (see Appendix~\ref{appendix_B}). For bounding $\beta_\pi$, note that $\lvert\pi\rvert\leq(8H_1(1+4H_1))^{-1}$ by the assumption of Theorem~\ref{theorem:estimate_rho_approx}, that is $\beta_\pi$ can be bounded by a multiple of $4H_1+16H_1^2$.
	\end{proof}
	
	\begin{remark}
		Carefully revisiting the proof of Theorem~\ref{theorem:main} one notices that the $\lvert\pi\rvert^\alpha$-part results from the $u_{xx}$-part of $F^\rho$. If $\sigma$ is independent of $Y_t$, this $u_{xx}$-part vanishes and thus, in this case, the convergence in time holds up to order $1$ instead of order $\alpha$.
	\end{remark}
	
	\section{Sufficient Conditions on the Coefficient Functions}\label{sec:sufficient_conditions}
	So far, the convergence does not fully rely on assumptions on the known functions $b$, $\sigma$, $F$, $g$ and $\rho$, but also on assumptions on the unknown solution~$u$ of the parabolic PDE~\eqref{eq:PDE} together with its spatial derivatives $u_x$ and $u_{xx}$. The goal of this section is to give additional conditions on $b$, $\sigma$, $F$ and $g$, such that Assumption~\ref{ass:sol_PDE} is fulfilled. To this end, we make use of classical results on the well-posedness of parabolic Cauchy problems by \citeauthor{Ladyzenskaja_Solonnikov_Uralceva_1968}~\cite{Ladyzenskaja_Solonnikov_Uralceva_1968}.
	
	In the following we write
	\begin{align*}
		x&=(x_1,\ldots,x_d)^\top,&b(t,x,y,z)&=(b_1(t,x,y,z),\ldots,b_d(t,x,y,z))^\top,\\
		z&=(z_1,\ldots,z_d),&\sigma(t,x,y)&=(\sigma_{ij}(t,x,y))_{i,j=1,\ldots,d}.
	\end{align*}

	\begin{assumption}\label{ass:sufficient}
		\begin{itemize}
			\item[(i)]
				There exists an $\alpha\in(0,1)$ such that $g\in H^{2+\alpha}(D)$ for any bounded domain $D\subset\mathbb{R}^d$ and~$g$ is bounded.\footnote{For a definition of $H^{2+\alpha}(D)$, see \cite[7]{Ladyzenskaja_Solonnikov_Uralceva_1968}.}
			\item[(ii)]
				There exists a positive real constant $K_F\in\mathbb{R}_{>0}$ such that for all $(t,x)\in[0,T]\times\mathbb{R}^d$ it holds that
				\begin{align*}
					\lvert F(t,x,0,0)\rvert\leq K_F.
				\end{align*}
			\item[(iii)]
				For all $i,j,k\in\{1,\ldots,d\}$, the partial derivatives
				\begin{align*}
					\frac{\partial b_i}{\partial y},\;
					\frac{\partial b_i}{\partial z_k},\;
					\frac{\partial\sigma_{ij}}{\partial x_k},\;
					\frac{\partial\sigma_{ij}}{\partial y},\;
					\frac{\partial F}{\partial y},\;
					\frac{\partial F}{\partial z_k}
				\end{align*}	
				exist, whereby
				\begin{align*}
					\frac{\partial\sigma_{ij}}{\partial x_k}\quad\text{and}\quad
					\frac{\partial\sigma_{ij}}{\partial y}
				\end{align*}
				fulfill Hölder conditions in $t$, $x$ and $y$ with exponents $\frac{\alpha}{2}$, $\alpha$ and $\alpha$.
		\end{itemize}
	\end{assumption}
	
	\begin{theorem}\label{theorem:FBSDE-sol}
		Let Assumptions~\ref{ass:functions} and~\ref{ass:sufficient} be fulfilled. Then Assumption~\ref{ass:sol_PDE} holds true. 
	\end{theorem}
	\begin{proof}
		For $(t,x,y,z)\in[0,T]\times\mathbb{R}^d\times\mathbb{R}\times(\mathbb{R}^d)^*$, define the functions (we set $\bar{t}\coloneqq T-t$)
		\begin{align*}
			a_i(t,x,y,z)&\coloneqq\frac{1}{2}\sum_{k=1}^dz_k\sum_{l=1}^d\sigma_{il}(\bar{t},x,y)\sigma_{lk}(\bar{t},x,y),\quad i=1,\ldots,d,\\
			a(t,x,y,z)&\coloneqq\begin{multlined}[t]
				\sum_{k=1}^dz_kb_k(\bar{t},x,y,z\sigma(\bar{t},x,y))-F(\bar{t},x,y,z\sigma(\bar{t},x,y))\\
				+\frac{1}{2}\sum_{i=1}^d\sum_{k=1}^dz_k\sum_{l=1}^d\left(\frac{\partial\sigma_{il}(\bar{t},x,y)}{\partial x_i}\sigma_{lk}(\bar{t},x,y)+\sigma_{il}(\bar{t},x,y)\frac{\partial\sigma_{lk}(\bar{t},x,y)}{\partial x_i}\right)\\
				+\frac{1}{2}\sum_{i=1}^dz_i\sum_{k=1}^dz_k\sum_{l=1}^d\left(\frac{\partial\sigma_{il}(\bar{t},x,y)}{\partial y}\sigma_{lk}(\bar{t},x,y)+\sigma_{il}(\bar{t},x,y)\frac{\partial\sigma_{lk}(\bar{t},x,y)}{\partial y}\right),
			\end{multlined}\\
			a_{ij}(t,x,y,z)&\coloneqq\frac{1}{2}\sum_{l=1}^d\sigma_{il}(\bar{t},x,y)\sigma_{lj}(\bar{t},x,y),\quad i,j=1,\ldots,d,\\
			A(t,x,y,z)&\coloneqq\sum_{l=1}^dz_lb_l(\bar{t},x,y,z\sigma(t,x,y))-F(\bar{t},x,y,z\sigma(\bar{t},x,y)),\\
			\psi_0(x)&\coloneqq g(x).
		\end{align*}
		Note that $a$ is well-defined, since, by Assumption~\ref{ass:sufficient}(iii), the involved partial derivatives exist. By the same argument, for each $i,j\in\{1,\ldots,d\}$, all of the following partial derivatives exist:
		\begin{align*}
			\frac{\partial a_i(t,x,y,z)}{\partial x_j}&=\frac{1}{2}\sum_{k=1}^dz_k\sum_{l=1}^d\left(\frac{\partial\sigma_{il}(\bar{t},x,y)}{\partial x_j}\sigma_{lk}(\bar{t},x,y)+\sigma_{il}(\bar{t},x,y)\frac{\partial\sigma_{lk}(\bar{t},x,y)}{\partial x_j}\right),\\
			\frac{\partial a_i(t,x,y,z)}{\partial y}&=\frac{1}{2}\sum_{k=1}^dz_k\sum_{l=1}^d\left(\frac{\partial\sigma_{il}(\bar{t},x,y)}{\partial y}\sigma_{lk}(\bar{t},x,y)+\sigma_{il}(\bar{t},x,y)\frac{\partial\sigma_{lk}(\bar{t},x,y)}{\partial y}\right),\\
			\frac{\partial a_i(t,x,y,z)}{\partial z_j}&=\frac{1}{2}\sum_{l=1}^d\sigma_{il}(\bar{t},x,y)\sigma_{lj}(\bar{t},x,y).
		\end{align*}
		The above functions are defined in such a way that the identities
		\begin{gather*}
			A(t,x,y,z)=a(t,x,y,z)-\sum_{i=1}^d\frac{\partial a_i(t,x,y,z)}{\partial x_i}-\sum_{i=1}^d\frac{\partial a_i(t,x,y,z)}{\partial y}z_i,\\
			a_{ij}(t,x,y,z)=\frac{\partial a_i(t,x,y,z)}{\partial z_j},\quad i,j=1,\ldots,d,
		\end{gather*}
		hold. Therefore, we are in the setting of \cite[Chapter~V, Theorem~8.1]{Ladyzenskaja_Solonnikov_Uralceva_1968}, which gives us sufficient conditions for the existence of a classical solution $\tilde{u}\in H^{1+\frac{\alpha}{2},2+\alpha}([0,T]\times\mathbb{R}^d)$\footnote{For a definition of $H^{1+\frac{\alpha}{2},2+\alpha}([0,T]\times\mathbb{R}^d)$, see \cite[7--8]{Ladyzenskaja_Solonnikov_Uralceva_1968}.} of the Cauchy problem
		\begin{align}\label{eq:Cauchy_problem}
			\begin{cases}
				\tilde{u}_t(t,x)=\sum_{i=1}^d\frac{\mathrm{d}}{\mathrm{d}x_i}a_i(t,x,\tilde{u}(t,x),\tilde{u}_x(t,x))-a(t,x,\tilde{u}(t,x),\tilde{u}_x(t,x)),\quad(t,x)\in[0,T]\times\mathbb{R}^d,\\
				\tilde{u}(0,x)=\psi_0(x),\quad x\in\mathbb{R}^d.
			\end{cases}
		\end{align}
		Note that a classical solution $\tilde{u}\in H^{1+\frac{\alpha}{2},2+\alpha}([0,T]\times\mathbb{R}^d)$ fulfills all boundedness, Hölder and Lipschitz conditions of Assumption~\ref{ass:sol_PDE}.
		
		In the following, we show that the above Cauchy problem~\eqref{eq:Cauchy_problem} admits a classical solution $\tilde{u}\in H^{1+\frac{\alpha}{2},2+\alpha}([0,T]\times\mathbb{R}^d)$ by verifying the solvability conditions (a), (b), (c) of \cite[Chapter~V, Theorem~8.1]{Ladyzenskaja_Solonnikov_Uralceva_1968}. In a final step, from $\tilde{u}$, we construct a solution $u$ of the parabolic PDE~\eqref{eq:PDE} which inherits the desired boundedness, Hölder and Lipschitz conditions.
		
		Solvability condition (a) coincides with Assumption~\ref{ass:sufficient}(i). Solvability condition (b) follows from the uniform ellipticity condition~\eqref{eq:elliptic} and the inequality
		\begin{align*}
					\lvert A(t,x,y,0)y\rvert
					&=\lvert F(t,x,y,0)y\rvert\\
					&=\lvert(F(t,x,y,0)-F(t,x,0,0))y+F(t,x,0,0)y\rvert\\
					&\leq H_F\lvert y\rvert^2+K_F\lvert y\rvert\\
					&\leq\left(H_F+\frac{1}{2}\right)\lvert y\rvert^2+\frac{K_F^2}{2},
				\end{align*}
		which uses the Lipschitz condition on $F$, Assumption~\ref{ass:sufficient}(ii) and Young's inequality and holds true for arbitrary $(t,x,y,z)\in[0,T]\times\mathbb{R}^d\times\mathbb{R}\times(\mathbb{R}^d)^*$.
		
		To verify solvability condition (c), first note that all functions
		\begin{align*}
			F,\;\sigma,\;b_k,\;\sigma_{ij},\;\frac{\partial\sigma_{ij}}{\partial x_k},\;\frac{\partial\sigma_{ij}}{\partial y},\qquad i,j,k\in\{1,\ldots,d\},
		\end{align*}
		are $\frac{1}{2}$-Hölder continuous in $t$ and Lipschitz continuous in $x$, $y$ and $z$ by Assumptions~\ref{ass:functions}(ii) and~\ref{ass:sufficient}(iii) and consequently continuous. Since additionally the mappings
		\begin{align*}
			(t,x,y,z)\mapsto\bar{t}=T-t
			\quad\text{and}\quad
			(t,x,y,z)\mapsto z_k,\quad k\in\{1,\ldots,d\},
		\end{align*}
		are obviously continuous, it follows that the functions
		\begin{align*}
			a_i,\;a,\;\frac{\partial a_i}{\partial x_j},\;\frac{\partial a_i}{\partial y},\;\frac{\partial a_i}{\partial z_j},\qquad i,j\in\{1,\ldots,d\},
		\end{align*}
		are continuous.
		
		Furthermore, for $\lvert z\rvert\leq K_z\in\mathbb{R}_{>0}$ it holds that
		\begin{align*}
			&\lvert b_k(\bar{t},x,y,z\sigma(\bar{t},x,y))-b_k(\bar{t}',x',y',z'\sigma(\bar{t}',x',y'))\rvert\\
			&\leq\lvert b(\bar{t},x,y,z\sigma(\bar{t},x,y))-b(\bar{t}',x',y',z'\sigma(\bar{t}',x',y'))\rvert\\
			&\leq H_b(1+K_zH_\sigma+K_\sigma)\bigl(\lvert t-t'\rvert^{1/2}+\lvert x-x'\rvert+\lvert y-y'\rvert+\lvert z-z'\rvert\bigr)
		\end{align*}
		and
		\begin{align*}
			&\lvert F(\bar{t},x,y,z\sigma(\bar{t},x,y))-F(\bar{t}',x',y',z'\sigma(\bar{t}',x',y'))\rvert\\
			&\leq H_b(1+K_zH_\sigma+K_\sigma)\bigl(\lvert t-t'\rvert^{1/2}+\lvert x-x'\rvert+\lvert y-y'\rvert+\lvert z-z'\rvert\bigr),
		\end{align*}
		i.\,e., the mappings $(t,x,y,z)\mapsto b_k(\bar{t},x,y,z\sigma(\bar{t},x,y))$ and $(t,x,y,z)\mapsto F(\bar{t},x,y,z\sigma(\bar{t},x,y))$ are $\frac{1}{2}$-Hölder continuous in $t$ and Lipschitz continuous in $x$, $y$ and $z$. The same also holds, obviously, for the projections $(t,x,y,z)\mapsto z_k$, $k\in\{1,\ldots,d\}$, and, by Assumptions~\ref{ass:functions}(ii) and~\ref{ass:sufficient}(iii), for the mappings
		\begin{align*}
			(t,x,y,z)\mapsto\sigma_{ij}(\bar{t},x,y),\;\frac{\partial\sigma_{ij}}{\partial x_k}(\bar{t},x,y),\;\frac{\partial\sigma_{ij}}{\partial y}(\bar{t},x,y),
			\quad
			i,j,k\in\{1,\ldots,d\},
		\end{align*}
		noting that $\lvert\bar{t}-\bar{t}'\rvert=\lvert t-t'\rvert$. It follows that the functions
		\begin{align*}
			a_i,\;a,\;\frac{\partial a_i}{\partial x_j},\;\frac{\partial a_i}{\partial y},\;\frac{\partial a_i}{\partial z_j},\qquad i,j\in\{1,\ldots,d\},
		\end{align*}
		fulfill the same Hölder and Lipschitz conditions as sum and product of Hölder and Lipschitz continuous functions where in the products all functions are bounded.
		Now, since $t\in[0,T]$, for $\lvert y\rvert\leq K_y\in\mathbb{R}_{>0}$ and $\lvert z\rvert\leq K_z$ it easily follows that these functions fulfill Hölder conditions in $t$, $y$, and $z$ with exponents $\frac{\alpha}{2}$, $\alpha$ and $\alpha$. From Assumptions~\ref{ass:functions}(iii) and~\ref{ass:sufficient}(ii) in connection with $\lvert z\rvert\leq K_z$ it follows that that each $f\in\bigl\{a_i,a,\frac{\partial a_i}{\partial x_j},\frac{\partial a_i}{\partial y},\frac{\partial a_i}{\partial z_j}\mid i,j\in\{1,\ldots,d\}\bigr\}$ can be bounded by some constant $K_f$ uniformly in $(t,x)$, that is for all $(t,x)\in[0,T]\times\mathbb{R}^d$ it holds that
		\begin{align*}
			\lvert f(t,x,0,0)\rvert\leq K_f.
		\end{align*}
		Consequently, each function $f\in\bigl\{a_i,a,\frac{\partial a_i}{\partial x_j},\frac{\partial a_i}{\partial y},\frac{\partial a_i}{\partial z_j}\mid i,j\in\{1,\ldots,d\}\bigr\}$ fulfills a Hölder condition of exponent $\alpha$ in $x$:
		\begin{align*}
			\lvert f(t,x,y,z)-f(t,x',y,z)\rvert
			\leq(H_b\vee 2(H_fK_yK_z+K_f))\lvert x-x'\rvert^\alpha
		\end{align*}
		for all $(t,y,z)\in[0,T]\times\mathbb{R}\times(\mathbb{R}^d)^*$ and any two $x,x'\in\mathbb{R}^d$.
		Lastly, from Assumptions~\ref{ass:functions}(ii)--(iii) and~\ref{ass:sufficient}(ii), we obtain
		\begin{align*}
			&\sum_{i=1}^d\left(\lvert a_i(t,x,y,z)\rvert+\left\lvert\frac{\partial a_i(t,x,y,z)}{\partial y}\right\rvert\right)(1+\lvert z\rvert)+\sum_{i=1}^d\sum_{j=1}^d\left\lvert\frac{\partial a_i(t,x,y,z)}{\partial x_j}\right\rvert+\lvert a(t,x,y,z)\rvert\\
			&\leq\begin{multlined}[t][0.8\displaywidth]
				\sum_{i=1}^d\left(\frac{1}{2}d^2K_\sigma^2\lvert z\rvert+d^2H_\sigma K_\sigma\lvert z\rvert\right)(1+\lvert z\rvert)+\sum_{i=1}^d\sum_{j=1}^dd^2H_\sigma K_\sigma\lvert z\rvert\\
				+dK_b\lvert z\rvert+K_F+H_F(\lvert y\rvert+K_\sigma\lvert z\rvert)+d^3H_\sigma K_\sigma\lvert z\rvert+d^3H_\sigma K_\sigma\lvert z\rvert^2,
			\end{multlined}
		\end{align*}
		i.\,e., for bounded $y$, the LHS admits a bound of the form $\mu(1+\lvert z\rvert)^2$ for a suitable constant $\mu\in\mathbb{R}_{>0}$. Together with the uniform ellipticity condition~\eqref{eq:elliptic}, solvability condition (c) is fulfilled.
		
		Now, noting that
		\begin{align*}
			\frac{\mathrm{d}}{\mathrm{d}x_i}a_i(t,x,\tilde{u}(t,x),\tilde{u}_x(t,x))
			&=\begin{multlined}[t]
				\frac{\partial a_i(t,x,\tilde{u}(t,x),\tilde{u}_x(t,x))}{\partial x_i}+\frac{\partial a_i(t,x,\tilde{u}(t,x),\tilde{u}_x(t,x))}{\partial y}\tilde{u}_{x_i}(t,x)\\
			+\sum_{k=1}^n\frac{\partial a_i(t,x,\tilde{u}(t,x),\tilde{u}_x(t,x))}{\partial z_k}\tilde{u}_{x_ix_k}(t,x)
			\end{multlined}
		\end{align*}
		and a classical solution $\tilde{u}$ fulfills $\tilde{u}_{x_ix_k}(t,x)=\tilde{u}_{x_kx_i}(t,x)$ for all $i,k\in\{1,\ldots,d\}$, from the definition of $a$ it follows that
		\begin{align*}
			\eqref{eq:Cauchy_problem}
			&\Leftrightarrow
				\left\{\begin{aligned}
					\tilde{u}_t(t,x)&=\frac{1}{2}\mathrm{tr}\bigl[\tilde{u}_{xx}(t,x)(\sigma\sigma^\top)(\bar{t},x,\tilde{u}(t,x))\bigr]-F(\bar{t},x,\tilde{u}(t,x),\tilde{u}_x(t,x)\sigma(\bar{t},x,\tilde{u}(t,x)))\\
					&\quad+\tilde{u}_x(t,x)b(\bar{t},x,\tilde{u}(t,x),\tilde{u}_x(t,x)\sigma(\bar{t},x,
					\tilde{u}(t,x))),\quad(t,x)\in[0,T]\times\mathbb{R}^d,\\
					\tilde{u}(0,x)&=g(x),\quad x\in\mathbb{R}^d,
				\end{aligned}\right.\\
			&\Leftrightarrow
				\left\{\begin{aligned}
					\tilde{u}_t(\bar{t},x)&=F(t,x,\tilde{u}(\bar{t},x),\tilde{u}_x(\bar{t},x)\sigma(t,x,\tilde{u}(\bar{t},x)))-\frac{1}{2}\mathrm{tr}\bigl[\tilde{u}_{xx}(\bar{t},x)(\sigma\sigma^\top)(t,x,\tilde{u}(\bar{t},x))\bigr]\\
					&\quad-\tilde{u}_x(\bar{t},x)b(t,x,\tilde{u}(\bar{t},x),\tilde{u}_x(\bar{t},x)\sigma(t,x,\tilde{u}(\bar{t},x))),\quad(t,x)\in[0,T]\times\mathbb{R}^d,\\
					\tilde{u}(\bar{T},x)&=g(x),\quad x\in\mathbb{R}^d,
			\end{aligned}\right.
		\end{align*}
		Consequently, we can define a desired solution $u$ of the parabolic PDE~\eqref{eq:PDE} via
		\begin{align*}
			u&\colon[0,T]\times\mathbb{R}^d\to\mathbb{R},\quad u(t,x)\coloneqq\tilde{u}(\bar{t},x).\qedhere
		\end{align*}	
	\end{proof}
		
	\printbibliography

	\appendix
	\renewcommand*{\sectionformat}{%
	  Appendix~\thesection:\enskip
	}
	\section{Proof of Lemma~\ref{lem:Euler_Scheme}}\label{appendix_A}
	From the recursion of the Euler-Maruyama scheme~\eqref{eq:Euler_extented}, we obtain
	\begin{align*}
		\tilde{X}^{\eta,\pi}_{t_i}&=\tilde{x}_0+\sum_{j=0}^{i-1}\bigl[b^\rho(t_j,\tilde{X}^{\eta,\pi}_{t_j})\Delta_j+\sigma^\rho(t_j,\tilde{X}^{\eta,\pi}_{t_j})\Delta W_j\bigr].
	\end{align*}
	It directly follows that, for all $t\in[0,T]$, we can express $\tilde{X}^{\eta,\pi}_t$ as
	\begin{align}\label{eq:recursion_extendend_integral}
		\tilde{X}^{\eta,\pi}_t=\tilde{x}_0+\int_0^{\Pi(t)}b^\rho(\Pi(s),\tilde{X}^{\eta,\pi}_s)\,\mathrm{d}s+\int_0^{\Pi(t)}\sigma^\rho(\Pi(s),\tilde{X}^{\eta,\pi}_s)\,\mathrm{d}W_s.
	\end{align}
	Let $t\in[0,T]$. Inserting the definition of $\tilde{X}^\eta_t$ and Equation~\eqref{eq:recursion_extendend_integral} for $\tilde{X}^{\eta,\pi}_t$, we obtain
	\begin{align*}
		&\mathbb{E}\bigl[\lvert\tilde{X}^\eta_t-\tilde{X}^{\eta,\pi}_t\rvert^2\bigr]\\
		&\begin{multlined}[t]
				\leq4\Biggl(\mathbb{E}\Biggl[\biggl\lvert\int_0^{\Pi(t)}b^\rho(s,\tilde{X}^\eta_s)-b^\rho(\Pi(s),\tilde{X}^{\eta,\pi}_s)\,\mathrm{d}s\biggr\rvert^2\Biggr]+\mathbb{E}\Biggl[\biggl\lvert\int_{\Pi(t)}^tb^\rho(s,\tilde{X}^\eta_s)\,\mathrm{d}s\biggr\rvert^2\Biggr]\\
				+\mathbb{E}\Biggl[\biggl\lvert\int_0^{\Pi(t)}\sigma^\rho(s,\tilde{X}^\eta_s)-\sigma^\rho(\Pi(s),\tilde{X}^{\eta,\pi}_s)\,\mathrm{d}W_s\biggr\rvert^2\Biggr]
				+\mathbb{E}\Biggl[\biggl\lvert\int_{\Pi(t)}^t\sigma^\rho(s,\tilde{X}^\eta_s)\,\mathrm{d}W_s\biggr\rvert^2\Biggr]\Biggr)
			\end{multlined}\\
		&\eqqcolon 4\bigl((\mathrm{I})+(\mathrm{II})+(\mathrm{III})+(\mathrm{IV})\bigr).
	\end{align*}
	Using Hölder's inequality, the $\frac{1}{2}$-Hölder and Lipschitz continuity of $b^\rho$ and Fubini's theorem, the first term can be bounded by
	\begin{align*}
		(\mathrm{I})
		&\leq\Pi(t)\mathbb{E}\biggl[\int_0^{\Pi(t)}\lvert b^\rho(s,\tilde{X}^\eta_s)-b^\rho(\Pi(s),\tilde{X}^{\eta,\pi}_s)\rvert^2\,\mathrm{d}s\biggr]\\
		&\leq TH_{b^\rho}^2\mathbb{E}\biggl[\int_0^{\Pi(t)}\bigl(\lvert s-\Pi(s)\rvert^{1/2}+\lvert \tilde{X}^\eta_s-\tilde{X}^{\eta,\pi}_s\rvert\bigr)^2\,\mathrm{d}s\biggr]\\
		&\leq2TH_{b^\rho}^2\biggl(\mathbb{E}\biggl[\int_0^{\Pi(t)}\lvert s-\Pi(s)\rvert\,\mathrm{d}s\biggr]+\mathbb{E}\biggl[\int_0^{\Pi(t)}\lvert \tilde{X}^\eta_s-\tilde{X}^{\eta,\pi}_s\rvert^2\,\mathrm{d}s\biggr]\biggr)\\
		&\leq2TH_{b^\rho}^2\biggl(T\lvert\pi\rvert+\int_0^t\mathbb{E}\bigl[\lvert\tilde{X}^\eta_s-\tilde{X}^{\eta,\pi}_s\rvert^2\bigr]\,\mathrm{d}s\biggr).
	\end{align*}
	For $(\mathrm{III})$, using Itô's isometry, the Hölder and Lipschitz condition on $\sigma^\rho$ and Fubini's theorem, we obtain
	\begin{align*}
		(\mathrm{III})
		&=\mathbb{E}\biggl[\int_0^{\Pi(t)}\lvert\sigma^\rho(s,\tilde{X}^\eta_s)-\sigma^\rho(\Pi(s),\tilde{X}^{\eta,\pi}_s)\rvert^2\,\mathrm{d}s\biggr]\\
		&\leq H_{\sigma^\rho}^2\mathbb{E}\biggl[\int_0^{\Pi(t)}\bigl(\lvert s-\Pi(s)\rvert^{1/2}+\lvert\tilde{X}^\eta_s-\tilde{X}^{\eta,\pi}_s\rvert\bigr)^2\,\mathrm{d}s\biggr]\\
		&\leq2H_{\sigma^\rho}^2\mathbb{E}\biggl[\int_0^{\Pi(t)}\lvert s-\Pi(s)\rvert+\lvert\tilde{X}^\eta_s-\tilde{X}^{\eta,\pi}_s\rvert^2\,\mathrm{d}s\biggr]\\
		&\leq2H_{\sigma^\rho}^2\biggl(T\lvert\pi\rvert+\int_0^t\mathbb{E}\bigl[\lvert\tilde{X}^\eta_s-\tilde{X}^{\eta,\pi}_s\rvert^2\bigr]\,\mathrm{d}s\biggr).
	\end{align*}
	Applying Hölder's inequality and noticing that $b^\rho$ fulfills the linear growth condition of Lemma~\ref{lem:linear_growth}, we see that the second term fulfills
	\begin{align*}
		(\mathrm{II})
		\leq\lvert\pi\rvert L_{b^\rho}^2\mathbb{E}\biggl[\int_{\Pi(t)}^t1+\lvert\tilde{X}^\eta_s\rvert^2\,\mathrm{d}s\biggr]
		\leq\lvert\pi\rvert^2L_{b^\rho}^2(1+c_6),
	\end{align*}
	where the last inequality follows from Lemma~\ref{lem:supremum-Xtilde}.
	Similarly, now using Itô's isometry and boundedness of $\sigma^\rho$ from Lemma~\ref{lem:linear_growth}, for (IV) we obtain
	\begin{align*}
		(\mathrm{IV})
		\leq\mathbb{E}\biggl[\int_{\Pi(t)}^tK_{\sigma^\rho}^2\,\mathrm{d}s\biggr]
		\leq K_{\sigma^\rho}^2\lvert\pi\rvert.
	\end{align*}
	Putting all bounds together yields
	\begin{multline*}
		\mathbb{E}\bigl[\lvert\tilde{X}^\eta_t-\tilde{X}^{\eta,\pi}_t\rvert^2\bigr]
		\leq8(TH_{b^\rho}^2+H_{\sigma^\rho}^2)\int_0^t\mathbb{E}\bigl[\lvert\tilde{X}^\eta_s-\tilde{X}^{\eta,\pi}_s\rvert^2\bigr]\,\mathrm{d}s\\
		+4\bigl(2T(TH_{b^\rho}^2+H_{\sigma^\rho}^2)+L_{b^\rho}^2\lvert\pi\rvert(1+c_6)+K_{\sigma^\rho}^2\bigr)\lvert\pi\rvert.
	\end{multline*}
	Applying Gronwall's lemma finishes the proof.

	\section{Upper constant bounds depending on \texorpdfstring{$K_\rho$}{Kρ} and \texorpdfstring{$H_\rho$}{Hρ}}\label{appendix_B}
	\begin{longtable}{
		*{7}{|>{$}c<{$}}|}
		\hline
		& C(1+K_\rho^p) & C(1+H_\rho) & C(1+K_\rho+H_\rho) & Ce^{CK_\rho^p} & Ce^{C(K_\rho^2+H_\rho^2)}\\
		\hline\endfirsthead
						
		\hline
		& C(1+K_\rho^p) & C(1+H_\rho) & C(1+K_\rho+H_\rho) & Ce^{CK_\rho^p} & Ce^{C(K_\rho^2+H_\rho^2)}\\
		\hline\endhead

		\hline
    		\endfoot
    		
    		\hline
    		\endlastfoot
		
		L_{b^\rho}				& \crossone & & & & \\
		K_{\sigma^\rho}			& \crossone & & & & \\
		H_{b^\rho}				& & \times & & & \\
		H_{\sigma^\rho}			& & \times & & & \\
		H_\zeta^{(t,x)}			& \crossone & & & & \\
		H_\zeta^{(y,z)}			& \crossone & & & & \\
		H_1						& \crossone & & & & \\
		H_2						& & & \times & & \\
		\hdashline
		\alpha_\pi				& \crosstwo & & & & \\
		\beta_\pi				& \crosstwo & & & & \\
		\Gamma					& \crossfour & & & & \\
		d_1						& & & & \crossfour & \\
		d_2						& & & & \crosstwo & \\
		d_3						& & & & \crosstwo & \\
		d_4						& & & & \crosstwo & \\
		d_5						& & & & \crosstwo & \\
		\tilde{d}_2				& & & & \crosstwo & \\
		\tilde{d}_3				& & & & \crosstwo & \\
		\hdashline
		c_1						& & & & \crossfour & \\
		c_2						& & & & \crossfour & \\
		c_3						& & & & & \times \\
		c_4						& & & & & \times \\
		c_5						& & & & & \times \\
		c_6						& & & & \crosstwo & \\
		c_7						& & & & \crossfour & \\
		c_8						& & & & \crossfour & \\
		c_9						& & & & & \times \\
		\tilde{c}_4				& & & & & \times \\
		\tilde{c}_5				& & & & & \times \\
		\tilde{c}_9				& & & & & \times \\
	\end{longtable}
	
	For all other constants, including $H_3$, $c_{10}$ and $c_{11}$, there exist upper bounds that do not depend on $K_\rho$ or $H_\rho$.
\end{document}